\documentclass[onefignum,onetabnum]{siamonline220329}
\usepackage{graphicx}
\usepackage[utf8]{inputenc}
\usepackage{amssymb}
\usepackage{cleveref}
\usepackage{algorithm}
\usepackage{algpseudocode}
\usepackage{bm}
\algrenewcommand\algorithmicrequire{\textbf{Input:}}
\algrenewcommand\algorithmicensure{\textbf{Output:}}
\usepackage{subfigure}
\newcommand{\norm}[1]{\left\lVert#1\right\rVert}
\newcommand{\bfA}{{\bf A}}
\newcommand{\bfB}{{\bf B}}

\newcommand{\bfC}{{\bf C}}

\newcommand{\bfD}{{\bf D}}

\newcommand{\bfF}{{\bf F}}

\newcommand{\bfG}{{\bf G}}
\newcommand{\bfH}{{\bf H}}

\newcommand{\bfI}{{\bf I}}

\newcommand{\bfK}{{\bf K}}
\newcommand{\bfL}{{\bf L}}
\newcommand{\bfM}{{\bf M}}
\newcommand{\tbfM}{\tilde{\bf M}}

\newcommand{\bfP}{{\bf P}}
\newcommand{\bfQ}{{\bf Q}}
\newcommand{\bfR}{{\bf R}}

\newcommand{\bfT}{{\bf T}}
\newcommand{\bfU}{{\bf U}}
\newcommand{\tbfU}{\tilde{\bfU}}
\newcommand{\bfV}{{\bf V}}

\newcommand{\bfW}{{\bf W}}

\newcommand{\bfX}{{\bf X}}

\newcommand{\bfY}{{\bf Y}}

\newcommand{\bfZ}{{\bf Z}}

\newcommand{\bfb}{{\bf b}}

\newcommand{\bfd}{{\bf d}}
\newcommand{\bfe}{{\bf e}}

\newcommand{\bfg}{{\bf g}}
\newcommand{\bfh}{{\bf h}}

\newcommand{\bfm}{{\bf m}}

\newcommand{\bfu}{{\bf u}}

\newcommand{\bfw}{{\bf w}}
\newcommand{\bfx}{{\bf x}}

\newcommand{\bfz}{{\bf z}}
\newcommand{\bfzero}{{\bf0}}
\newcommand{\bfone}{{\bf1}}

\newcommand{\bfUpsilon}{{\boldsymbol{\Upsilon}}}

\newcommand{\tbfUpsilon}{\tilde{\boldsymbol{\Upsilon}}}
\newcommand{\bfPhi}{{\boldsymbol{\Phi}}}
\newcommand{\tbfPhi}{\tilde{\boldsymbol{\Phi}}}

\newcommand{\tbfPsi}{\tilde{\boldsymbol{\Psi}}}

\DeclareMathOperator*{\argmin}{arg\,min}
\newcommand*\samethanks[1][\value{footnote}]{\footnotemark[#1]}

\title{Efficient Decomposition-Based Algorithms for $\ell_1$-Regularized Inverse Problems with Column-Orthogonal and Kronecker Product Matrices}

\author{Brian Sweeney\thanks{School of Mathematical and Statistical Sciences, Arizona State University, Tempe, AZ (\href{mailto:bfsweene@asu.edu}{bfsweene@asu.edu},  \href{mailto:malena.espanol@asu.edu}{malena.espanol@asu.edu},  and \href{mailto:renaut@asu.edu}{renaut@asu.edu} )} \and Malena I. Espa\~nol\samethanks \and Rosemary Renaut\samethanks} 
\headers{GSVD for Kronecker Products in $\ell_1$ regularization}{Sweeney,  Espa\~nol, and Renaut}
\begin{document} 

\maketitle

\begin{abstract}
We consider an $\ell_1$-regularized inverse problem where both the forward and regularization operators have a Kronecker product structure. By leveraging this structure, a joint decomposition can be obtained using generalized singular value decompositions. We show how this joint decomposition can be effectively integrated into the Split Bregman and Majorization-Minimization methods to solve the 
$\ell_1$-regularized inverse problem. Furthermore, for cases involving column-orthogonal regularization matrices, we prove that the joint decomposition can be derived directly from the singular value decomposition of the system matrix. As a result, we show that framelet and wavelet operators are efficient for these decomposition-based algorithms in the context of $\ell_1$-regularized image deblurring problems. 
\end{abstract}



\section{Introduction}
We are interested in solving discrete ill-posed inverse problems of the form $\bfA \bfx \approx \bfb$, where $\bfA \in \mathbb{R}^{m\times n}$, $m \geq n$, $\bfb \in \mathbb{R}^m$, and $\bfx \in \mathbb{R}^n$.  
 It is assumed that the matrix $\bfA$ has full column rank but is ill-conditioned, and $\bfb$ is contaminated by additive Gaussian noise: $\bfb = \bfA\bfx + \bfe$, where $\bfe$ is a Gaussian noise vector. Throughout this paper, we normalize the problem by dividing $\bfA$ and $\bfb$ by the standard deviation of $\bfe$. This corresponds to whitening the noise in the data.
 Due to the ill-posedness of this problem, the solution of the normal equations, $\bfx = (\bfA^\top\bfA)^{-1}\bfA^\top\bfb$, is poor, and therefore, we will apply regularization.  With generalized Tikhonov regularization~\cite{tikhonov1963solution}, $\bfx$ is selected to solve the minimization problem
\begin{align}\label{eq:GenTik}
    \min_{\bfx}\left\{ \frac{1}{2}\norm{\bfA \bfx - \bfb}_2^2 + \frac{\lambda^2}{2} \norm{\bfL \bfx}_2^2\right\},
\end{align}
where $\lambda$ is a regularization parameter and $\bfL \in \mathbb{R}^{p\times n}$ is a regularization matrix.  $\bfL$ is often selected as the discretization of some derivative operator \cite{hansen2010discrete}. We also assume that 
\begin{align}\label{eq:nullprop}
    \mbox{Null}(\bfA) \cap \mbox{Null}(\bfL) = \emptyset,
\end{align} 
where $\mbox{Null}(\bfC)$ denotes the null space of the matrix $\bfC$.

Instead of considering the Tikhonov problem \cref{eq:GenTik}, we will consider a sparsity-preserving problem
\begin{align}\label{eq:Lxh}
    \min_\bfx\left\{\frac{1}{2}\norm{\bfA\bfx-\bfb}_2^2 + \mu\norm{\bfL\bfx}_1\right\},
\end{align}
where $\mu$ is a regularization parameter. We will solve \cref{eq:Lxh} using two iterative methods, split Bregman (SB)~\cite{GO} and Majorization-Minimization (MM)~\cite{ huang2017majorization, hunter2004tutorial}. Both methods share a minimization sub-problem of the same form~\cite{sweeney2024parameter}.

Our focus is on  Kronecker Product (KP) matrices $\bfA$ and $\bfL$, that is, matrices that can be written as $\bfA = \bfA_1 \otimes \bfA_2$ and $\bfL = \bfL_1 \otimes \bfL_2,$ where $\otimes$ denotes the Kronecker product, which is defined for matrices $\bfC \in \mathbb{R}^{m \times n}$ and $\bfD \in \mathbb{R}^{p \times q}$ by
\begin{equation}\label{eq:defineKP}
    \bfC \otimes \bfD = \begin{bmatrix}
        c_{11}\bfD & c_{12} \bfD & \cdots & c_{1n} \bfD \\
        c_{21}\bfD & c_{22} \bfD & \cdots & c_{2n} \bfD \\
        \vdots & \vdots && \vdots \\
        c_{m1}\bfD & c_{m2} \bfD & \cdots & c_{mn} \bfD
    \end{bmatrix}_{mp \times nq}.
\end{equation}
Thus, we are interested in solving problem \cref{eq:Lxh} of the specific form
\begin{align}\label{eq:Lxhsep}
    \min_\bfx\left\{\frac{1}{2}\norm{(\bfA_1\otimes\bfA_2)\bfx-\bfb}_2^2 + \mu\norm{(\bfL_1\otimes\bfL_2)\bfx}_1\right\}.
\end{align}

KP matrices $\bfA$ and $\bfL$ arise in many applications. In image deblurring problems, $\bfA$ is a KP matrix, provided that the blur can be separated in the horizontal and vertical directions~\cite{hansen2006deblurring}.  Another case where $\bfA$ is a KP matrix is in 2D nuclear magnetic resonance (NMR) problems~\cite{bortolotti2016uniform,gazzola2019ir}, where the goal is to reconstruct a map of relaxation times from a probed material.
$\bfL$ is a KP matrix with certain forms of regularization, such as when orthonormal framelets \cite{ron1997affine} or wavelets \cite{daubechies1992ten} are used in two dimensions. Both are tight frames, namely, they form redundant orthogonal bases for $\mathbb{R}^n$. This redundancy allows for robust representations where the loss of information is more tolerable \cite{cai2010split}.
When used as regularization in \cref{eq:Lxh}, framelet \cite{buccini2020modulus,cai2009linearized,cai2010split,wang2014ultrasound} and wavelet-based \cite{fang2014wavelet} methods enforce sparsity on the solution within the tight frame.

In \cite{buccini2020modulus}, tight frames are used for regularization in a general $\ell_p$-$\ell_q$ problem without the assumptions of having KP matrices. 
Other minimization problems with KP matrices $\bfA$ or $\bfL$ have been studied before. Least-squares problems of the form
\begin{align*}
    \min_\bfx \norm{(\bfA_1 \otimes \bfA_2)\bfx - \bfb}^2_2,
\end{align*}
have been solved efficiently using the singular value and QR decompositions \cite{fausett1994large,fausett1997improved}. The special case where the forward operator has the form $\bfA = \bfI \otimes \bfA_2$, which is equivalent to a matrix inverse problem, has also been analyzed \cite{greensite2005inverse}. 
For least-squares problems of the form 
\begin{align}\label{eq:quickGSVD}
    \min_\bfx \norm{\begin{bmatrix}\bfA_1 \otimes \bfA_2 \\ \bfB_1 \otimes \bfB_2 \end{bmatrix}\bfx - \bfb}^2_2,
\end{align}
the generalized singular value decompositions (GSVDs) of $\{\bfA_1,\bfB_1\}$ and $\{\bfA_2,\bfB_2\}$ can be used to obtain a joint decomposition and solve the problem \cite{bardsley2011structured,knepper2011large,van2000ubiquitous}.  In particular, Givens rotations can be used to diagonalize the matrix containing the generalized singular values, which then diagonalizes the problem. A regularized problem of the form
\begin{align}\label{eq:quickGSVD2}
    \min_\bfx \norm{\begin{bmatrix}\bfA_x \otimes \bfL_y \\ \bfL_x \otimes \bfA_y \end{bmatrix}\bfx - \bfb}^2_2+\lambda^2\norm{\begin{bmatrix}\bfI \otimes \bfL_y \\ \bfL_x \otimes \bfI \end{bmatrix}\bfx}_2^2
\end{align}
occurs in adaptive optics \cite{bardsley2011structured,knepper2011large}. The structure of this problem can be utilized to rewrite it as a least-squares problem with a preconditioner.
Another problem that utilizes KP matrices is the constrained least-squares problem
\begin{align*}
    \min_\bfx \norm{(\bfA_1 \otimes \bfA_2)\bfx - \bfb}^2_2, \quad \text{s.t. } (\bfL_1 \otimes \bfL_2)\bfx = \bfh,
\end{align*}
which is considered in \cite{barrlund1998efficient}.  In this case, both the forward operator and the constraint are KP matrices, and a null space method can be applied to the matrix equations.  The solution is unique under certain circumstances, which can be analyzed using the GSVD~\cite{barrlund1998efficient}.  

In cases where $\bfA$ is not a KP matrix, there are methods to approximate it by a Kronecker product $\bfA \approx \bfC \otimes \bfD$ so that the benefits of the KP structure can still be utilized. The nearest Kronecker product can be found by minimizing
\begin{align*}
    \norm{\bfA - \bfC \otimes \bfD}_F^2.
\end{align*}
Van Loan and Pitsianis \cite{van1993approximation} developed a method for solving this problem, and in~\cite{pitsianis1997kronecker}, Pitsianis  extended it to solve the problem
\begin{align*}
    \norm{\bfA - \sum_{i=1}^k \bfC_i \otimes \bfD_i}_F^2.
\end{align*}
These KP approximations have been applied to the forward matrix $\bfA$ in imaging problems \cite{bouhamidi2012kronecker,garvey2018singular,kamm1998kronecker,kamm2000optimal}.

\textbf{Main Contributions: }
We utilize joint decompositions that exploit the KP structure, computed at the first iteration only, to solve \cref{eq:Lxhsep} with SB and MM where the regularization parameter is selected at each iteration as in \cite{sweeney2024parameter}. We introduce a family of regularization operators that are column orthogonal. This family includes framelet and wavelet matrices. We show that with these matrices, we can obtain the GSVDs needed for the joint decompositions directly from the singular value decompositions (SVDs) of $\bfA_1$ and $\bfA_2$, allowing us to use framelets and wavelets without computing a GSVD.  Other regularization matrices, such as finite-difference operators, would require computing GSVDs to implement these same decomposition-based algorithms.  This makes using framelet and wavelet operators efficient for solving $\ell_1$-regularized problems with SB and MM. We also show that the solution of the generalized Tikhonov problem with a fixed regularization parameter will generate the same solution using framelets and wavelets.  

This paper is organized as follows. In \cref{sec:Solver}, we review two iterative methods for solving \cref{eq:Lxh}, split Bregman and Majorization-Minimization.  In \cref{sec:GSVDKP}, the GSVD is presented and we discuss how it can be used to obtain a joint decomposition that can then be utilized within the iterative methods. Theoretical results are also presented for when $\bfL_1$ and $\bfL_2$ are column orthogonal, which allow this joint decomposition to be obtained from the SVD rather than the GSVD.  
In \cref{sec:regop}, we introduce a family of column-orthogonal regularization operators that include the framelet and wavelet regularization matrices. The iterative methods are then tested on image deblurring examples in \cref{sec:num}. Although the results presented in \cref{sec:GSVDKP} apply to a general $\bfA$, the examples are presented for cases with $m\geq n$. Conclusions are presented in \cref{sec:conc}.

\section{Iterative techniques for solution of the sparsity constrained problem}\label{sec:Solver}
In this section, we will consider solving \cref{eq:Lxh} using two iterative methods: the split Bregman (SB) method and the Majorization-Minimization (MM) method. Both methods only require computing two GSVDs, which can be reused for every iteration.

\subsection{Split Bregman}
The SB method was first introduced by Goldstein and Osher \cite{GO}. This method rewrites \cref{eq:Lxh} as a constrained optimization problem
\begin{align}\label{eq:l11}
    \min_{\bfx}\left\{\frac{1}{2}\|\bfA \bfx - \bfb \|^2_2 + \mu \|\bfd \|_1\right\} \quad \text{ s.t. } \bfL \bfx = \bfd,
\end{align}
which can then be converted to the unconstrained optimization problem
\begin{align} \label{eq:Breg1}
    \min_{\bfx,\bfd}\left\{\frac{1}{2}\|\bfA \bfx - \bfb \|^2_2 + \frac{\lambda^2}{2} \|\bfL \bfx - \bfd \|^2_2 + \mu \|\bfd \|_1\right\}.
\end{align}
This unconstrained optimization problem can be solved through a series of minimization steps and updates, here indicated by superscripts for iteration $k$,  known as the SB iteration
\begin{align}\label{eq:Breg2}
    (\bfx^{(k+1)}, \bfd^{(k+1)}) &= \argmin_{\bfx,\bfd}\left\{\frac{1}{2}\|\bfA \bfx - \bfb \|^2_2 + \frac{\lambda^2}{2}\|\bfL \bfx - \bfd^{(k)}+ \bfg^{(k)} \|^2_2 + \mu \|\bfd \|_1\right\} \\
    \bfg^{(k+1)} &= \bfg^{(k)} + (\bfL \bfx^{(k+1)} - \bfd^{(k+1)}). \label{eq:Breg2a}
\end{align}
In \cref{eq:Breg2}, the vectors $\bfx$ and $\bfd$ can be found separately as
\begin{align}\label{eq:Breg3}
    \bfx^{(k+1)} &= \argmin_\bfx \left\{\frac{1}{2}\|\bfA \bfx-\bfb\|_2^2 +\frac{\lambda^2}{2} \|\bfL\bfx -(\bfd^{(k)}-\bfg^{(k)})\|_2^2\right\}  \\ 
\bfd^{(k+1)}&= \argmin_\bfd \left\{\mu \|\bfd \|_1+\frac{\lambda^2}{2} \|\bfd-(\bfL\bfx^{(k+1)}+\bfg^{(k)})\|_2^2 \right\}. \label{eq:Breg4} 
\end{align}
Here, and in the update \cref{eq:Breg2a}, $\bfg$ is the vector of Lagrange multipliers. 
Defining $\tau = \mu/\lambda^2$, \cref{eq:Breg4} becomes
\begin{align}
    \bfd^{(k+1)}= \argmin_\bfd \left\{\tau \|\bfd \|_1+\frac{1}{2} \|\bfd-(\bfL\bfx^{(k+1)}+\bfg^{(k)})\|_2^2 \right\}. \label{eq:Breg5}
\end{align}
Since the elements of $\bfd$ are decoupled in \cref{eq:Breg5}, $\bfd^{(k+1)}$ can be computed using shrinkage operators. Consequently, each element is given by
\begin{align}
    d_j^{(k+1)} = \text{shrink}\left((\bfL\bfx^{(k+1)})_j+g^{(k)}_j,\tau\right),  \nonumber
\end{align}
where $\text{shrink}(x,\tau) = \text{sign}(x) \cdot \text{max}(|x|-\tau,0)$. We will also select the parameter $\lambda$ at each iteration by applying parameter selection methods to \cref{eq:Breg3} as in \cite{sweeney2024parameter}. We terminate iterations once the relative change in the solution $\bfx$, defined as
\begin{align}\label{eq:RCx}
    \text{RC}(\bfx^{(k+1)}) =\frac{\norm{\bfx^{(k+1)}-\bfx^{(k)}}_2}{\norm{\bfx^{(k)}}_2},
\end{align}
drops below a specified tolerance $tol$ or after a specified maximum number of iterations $K_{max}$ is reached.
The SB algorithm is summarized in \cref{alg:SB2}.
Notice that SB is related to applying the alternating direction method of multipliers (ADMM) to the augmented Lagrangian in \cref{eq:Breg1} \cite{esser2009applications}.  

\begin{algorithm}
\caption{The SB Method for the $\ell_2$-$\ell_1$ Problem \cref{eq:Lxh} }
\label{alg:SB2}
\begin{algorithmic}[1]
\Require $\bfA, \bfb, \bfL,\tau, \bfd^{(0)} = \bfg^{(0)} = {\bfzero}$, $tol$, $K_{max}$
\Ensure $\bfx$
    \For{$k=0,1,\dots$ until $\text{RC}(\bfx^{(k+1)}) < tol$ or $k = K_{max}$} 
    \State Estimate $\lambda^{(k)}$
    \State $\bfx^{(k+1)} = \argmin_\bfx \left\{\frac{1}{2}\|\bfA\bfx-\bfb\|_2^2 + \frac{\left(\lambda^{(k)}\right)^2}{2}\|\bfL\bfx-(\bfd^{(k)}-\bfg^{(k)})\|_2^2\right\}$
    \State $\bfd^{(k+1)} = \argmin_\bfd \left\{ \tau\|\bfd\|_1 + \frac{1}{2}\|\bfd-(\bfL\bfx^{(k+1)}+\bfg^{(k)})\|_2^2\right\}$
    \State $\bfg^{(k+1)} = \bfg^{(k)} + \left(\bfL\bfx^{(k+1)} - \bfd^{(k+1)}\right)$ 
    \EndFor
\end{algorithmic}
\end{algorithm}

\subsection{Majorization-Minimization}
The MM method is another iterative method for solving \cref{eq:Lxh}. It is an optimization method that utilizes two steps: majorization and minimization \cite{hunter2004tutorial}.  First, in the majorization step, the function is majorized with a surrogate convex function.  The convexity of this function is then utilized in the minimization step.  When applied to \cref{eq:Lxh}, MM is often combined with the generalized Krylov subspace (GKS) method \cite{buccini2020lp,huang2017majorization, pasha2020krylov} to solve large-scale problems. In this method, known as MM-GKS, the Krylov subspace is enlarged at each iteration, and MM is then applied to the problem in the subspace.  
In this paper, we will apply the MM method directly to \cref{eq:Lxh} instead of building a Krylov subspace.  For the majorant, we will use the fixed quadratic majorant from \cite{huang2017majorization}.  There are other ways for majorizing \cref{eq:Lxh} as well, including both fixed and adaptive quadratic majorants \cite{alotaibi2021restoration,buccini2020lp,pasha2020krylov}.  

The fixed majorant in \cite{huang2017majorization} depends on a parameter $\varepsilon > 0$. With this fixed majorant, the minimization problem at the $k^{\text{th}}$ iteration is
\begin{align} \label{eq:MMsolve}
    \min_{\bfx}\left\{\frac{1}{2}\|\bfA\bfx-\bfb\|^2_2 + \frac{\lambda^2}{2}\|\bfL\bfx-\bfw_{reg}^{(k)}\|^2_2\right\},
\end{align}
where $\lambda = \sqrt{\mu/\varepsilon}$ and
\begin{align}
    \bfw_{reg}^{(k)} = \bfu^{(k)}\left(1-\left(\frac{\varepsilon^2}{(\bfu^{(k)})^2+\varepsilon^2}\right)^{\frac{1}{2}}\right) \label{eq:wreg}
\end{align}
with $\bfu^{(k)} = \bfL \bfx^{(k)}$. All operations in \cref{eq:wreg} are component-wise.  As in \cite{sweeney2024parameter}, we will consider selecting the parameter $\lambda$ at each iteration using the minimization in \cref{eq:MMsolve}.  The iterations are again terminated when $\text{RC}(\bfx^{(k+1)})$ drops below a specified tolerance $tol$ or after $K_{max}$ iterations are reached. The MM algorithm is summarized in \Cref{alg:MML2}.

\begin{algorithm}
\caption{The MM Method for the $\ell_2$-$\ell_1$ Problem \cref{eq:Lxh} with a Fixed Quadratic Majorant}
\label{alg:MML2} 
\begin{algorithmic}[1]
\Require $\bfA, \bfb, \bfL, \bfx^{(0)}, \varepsilon$, $tol$, $K_{max}$
\Ensure $\bfx$
    \For{$k=0,1,\dots$ until $\text{RC}(\bfx^{(k+1)}) < tol$ or $k = K_{max}$} 
    \State $\bfu^{(k)} = \bfL\bfx^{(k)}$
    \State $\bfw_{reg}^{(k)} = \bfu^{(k)} \left(1-\left(\frac{\varepsilon^2}{\left(\bfu^{(k)}\right)^2+\varepsilon^2}\right)^{\frac{1}{2}}\right)$
    \State Estimate $\lambda^{(k)}$
    \State $\bfx^{(k+1)} = \argmin_{\bfx}\left\{ \frac{1}{2}\|\bfA\bfx-\bfb\|^2_2 + \frac{\left(\lambda^{(k)}\right)^2}{2}\|\bfL\bfx-\bfw_{reg}^{(k)}\|^2_2\right\}$
    \EndFor
\end{algorithmic}
\end{algorithm}

\subsection{Parameter Selection Methods}
In both iterative methods, we will solve a minimization problem of the form
\begin{align}\label{eq:Lxh2}
    \min_\bfx\left\{\frac{1}{2}\norm{\bfA\bfx-\bfb}_2^2 + \frac{\left(\lambda^{(k)}\right)^2}{2}\norm{\bfL\bfx-\bfh^{(k)}}_2^2\right\},
\end{align}
at each iteration.
To select $\lambda^{(k)}$, we will use both generalized cross validation (GCV) and the $\chi^2$ degrees of freedom (dof) test as discussed in the context of application to SB and MM algorithms in \cite{sweeney2024parameter}. GCV \cite{aster2018parameter, golub1979generalized} selects $\lambda$ to minimize the average predictive risk when an entry of $\bfb$ is left out.  This method does not require any information about the noise.  We will follow \cite{sweeney2024parameter} in which the GCV is extended to this inner minimization problem.  

The $\chi^2$ dof test \cite{mead2008parameter,mead2008newton,renaut2010regularization} is a statistical test that treats $\bfx$ as a random variable with mean $\bfx_0$ and utilizes information about the noise to select $\lambda$.  Under the assumption that $\bfe$ is normally distributed with mean ${\bf{0}}$ and covariance matrix $\bfI$, denoted $\bfe \sim \mathcal{N}({\bf{0}},\bfI)$, and $\bfL(\bfx - \bfx_0)\sim \mathcal{N}({\bf{0}},\lambda^{-2}\bfI)$, the functional
\begin{align*}\label{eq:chirls1}
   J(\bfx) = \|\bfA \bfx - \bfb\|_2^2 + \lambda^2\norm{\bfL(\bfx - \bfx_0)}_2^2
\end{align*}
follows a $\chi^2$ distribution with $\text{rank}(\bfL) + \max{(m-n,0)}$ dof \cite{sweeney2024parameter}.  The $\chi^2$ dof test then selects $\lambda$ so that $J$ most resembles this $\chi^2$ distribution. As an approximation of $\bfx_0$, we will use $\bfx_0 = \bfL_\bfA^\dagger\bfh^{(k)}$ as in \cite{sweeney2024parameter}. In this case, the determination of $\lambda$ is implemented efficiently as a one-dimensional root-finding algorithm. On the other hand, if $\bfx_0$ is not the mean of $\bfx$, then $J$ instead follows a non-central $\chi^2$ distribution \cite{renaut2010regularization} and a non-central $\chi^2$ test can be applied in a similar, but more difficult manner to select $\lambda$.

\section{Solving Generalized Tikhonov with Kronecker Product Matrices}\label{sec:GSVDKP}
Under the assumption that both $\bfA$ and $\bfL$ are KP matrices, we can solve the minimization problem \cref{eq:Lxh2} occurring for both the SB and MM methods by using a joint decomposition resulting from the individual GSVDs of 
$\{\bfA_1,\bfL_1\}$ and $\{\bfA_2,\bfL_2\}$. 
The GSVD is a joint decomposition of a matrix pair $\{\bfA,\bfL\}$ introduced by Van Loan \cite{van1976generalizing}. It has been formulated in different ways \cite{howland2004generalizing,paige1981towards}, but for this paper, we use the definition of the GSVD given in \cite{howland2004generalizing}. 
 \begin{lemma}\label{lem:shortgsvd}
 For $\bfA \in \mathbb{R}^{m \times n}$ and $\bfL \in \mathbb{R}^{p \times n}$, let $\bfK = \left[\bfA;\bfL\right]$. Let $t = \text{rank}(\bfK)$, $r = \text{rank}(\bfK) - \text{rank}(\bfL)$, and $s = \text{rank}(\bfA) + \text{rank}(\bfL) - \text{rank}(\bfK)$.    Then, there exist orthogonal matrices $\bfU   \in \mathbb{R}^{m \times m}$,  $\bfV   \in \mathbb{R}^{p \times p}$, $\bfW  \in \mathbb{R}^{t \times t}$, and $\bfQ  \in \mathbb{R}^{n \times n}$, and non-singular matrix $\bfR \in \mathbb{R}^{t \times t}$ such that 
\begin{equation}\label{eq:factorAL}
\bfU^\top\bfA \bfQ=\tbfUpsilon \left[\bfW^\top \bfR, \bfzero_{t \times (n-t)}\right] \quad \bfV^\top\bfL\bfQ=\tbfM \left[\bfW^\top \bfR, \bfzero_{t \times (n-t)}\right],
\end{equation}
where $\tbfUpsilon \in \mathbb{R}^{m \times t}$ and $\tbfM \in \mathbb{R}^{p \times t}$ with 
\begin{subequations}
\begin{align}
\tbfUpsilon&= \begin{bmatrix} \bfI_\bfA & & \\ & \bfUpsilon & \\ & & \bfzero_\bfA
\end{bmatrix}, \quad 
 \tbfM = \begin{bmatrix} \bfzero_\bfL & & \\ & \bfM & \\ & & \bfI_\bfL
\end{bmatrix} \label{eq:GSVD1} \\
\bfUpsilon&=\mathrm{diag}( \upsilon_{r+1}, \dots, \upsilon_{r+s}) \text{ with }
1>\upsilon_{r+1}\ge \upsilon_{r+2} \ge \dots \ge \upsilon_{r+s} > 0 \\
\bfM&=\mathrm{diag}(\mu_{r+1}, \dots, \mu_{r+s}) \text{ with } 
0< \mu_{r+1}\le \mu_{r+2} \le \dots \le \mu_{r+s} < 1,
\text{ and }  \\
 1&=\upsilon_i^2+\mu_i^2, i=r+1:r+s, \label{eq:GSVprop}
 \end{align}
 \end{subequations}
$\text{defining } \upsilon_i = 1, \; \mu_i=0 \text{ for } i< r+1, \text{ and } \upsilon_i=0, \; \mu_i=1, \text{ for } i> r+s.$ Here the identity matrices are $\bfI_\bfA \in \mathbb{R}^{r \times r} \text{ and }  \bfI_\bfL \in \mathbb{R}^{(t -r-s) \times (t -r-s)}$
and the zero matrices are $\bfzero_\bfA \in \mathbb{R}^{(m-r-s) \times (t-r-s)} \text{ and }  \bfzero_\bfL \in \mathbb{R}^{(p-t+r) \times r}.$
The singular values of $\bfR$ are the nonzero singular values of $\bfK$.
Defining
\begin{align*}
    \bfZ = \bfQ
    \begin{bmatrix}
        \bfR^\top \bfW \\ {\bf{0}}_{(n-t) \times t}
    \end{bmatrix},
\end{align*}
then $\bfA = \bfU \tbfUpsilon \bfZ^\top$ and $\bfL = \bfV \tbfM \bfZ^\top$, where $\bfZ \in \mathbb{R}^{n \times t}$ has full column rank. If $t=n$, then $\bfZ$ is invertible and we have that $\bfA = \bfU \tbfUpsilon \bfY^{-1}$ and $\bfL = \bfV \tbfM \bfY^{-1}$, where $\bfY^{-1} = \bfZ^\top$.
 \end{lemma}

The ratios $\gamma_i=\upsilon_i/\mu_i$ in the GSVD are known as the generalized singular values (GSVs) \cite{hansen1989regularization} and are sorted in decreasing order in \cref{lem:shortgsvd}.  If $\mu_i=0$ and $\upsilon_i=1$, then the corresponding GSV $\gamma_i$ is infinite.  The columns of $\bfU$ that correspond to $\upsilon_i = 0$ span $\mbox{Null}(\bfA)$ while the columns of $\bfV$ that correspond to $\mu_i = 0$ span $\mbox{Null}(\bfL)$.

From the GSVDs of the factors $\{\bfA_1,\bfL_1\}$ and $\{\bfA_2,\bfL_2\}$, 
 \begin{equation}\label{eq:2GSVDs}
 \begin{bmatrix}
 \bfA_1 \\
\bfL_1
 \end{bmatrix}  
 = 
\begin{bmatrix}
 \bfU_1\tbfUpsilon_1 \bfZ_1^\top  \bf\\
\bfV_1 \tbfM_1 \bfZ_1^\top
 \end{bmatrix} \quad \rm and \quad  \begin{bmatrix}
 \bfA_2 \\
\bfL_2
 \end{bmatrix}  
 = 
\begin{bmatrix}
 \bfU_2\tbfUpsilon_2 \bfZ_2^\top  \bf\\
\bfV_2 \tbfM _2 \bfZ_2^\top
 \end{bmatrix},
\end{equation}
we can construct a joint decomposition of $\bfA$ and $\bfL$ using the properties of the Kronecker product~\cite{van2000ubiquitous}.
 This joint decomposition is  given by
 \begin{equation}\label{eq:GSVDKP}
\bfK = 
\begin{bmatrix}
\bfA_1 \otimes \bfA_2 \\
\bfL_1 \otimes \bfL_2
\end{bmatrix}
=
\begin{bmatrix}
\bfU_1 \otimes \bfU_2 & \bfzero\\
\bfzero & \bfV_1 \otimes \bfV_2
\end{bmatrix}
\begin{bmatrix} 
\tbfUpsilon_1 \otimes \tbfUpsilon_2\\
\tbfM_1 \otimes \tbfM_2
\end{bmatrix}
\begin{bmatrix}
\bfZ_1^\top\otimes \bfZ_2^\top
\end{bmatrix}.
\end{equation}

Let us assume, without loss of generality, that $\bfx$ is square, meaning that $\text{array}(\bfx) = \bfX \in \mathbb{R}^{n \times n}$,  where $\text{array}(\bfx)$ is an operator that reshapes the vector $\bfx \in \mathbb{R}^{n^2 \times 1}$ into an $n \times n$ matrix.  We will also assume that $\bfA_1,\bfA_2,\bfL_1,\bfL_2$ have at least as many rows as columns.
With assumption \cref{eq:nullprop}, $t=n$ in \cref{lem:shortgsvd} and the solution to \cref{eq:Lxh2} is then given as
\begin{align}\label{eq:innersol}
    \bfx = (\bfA^\top\bfA + \lambda^2\bfL^\top \bfL)^{-1}(\bfA^\top\bfb + \lambda^2\bfL^\top\bfh^{(k)}).
\end{align}
Since $t=n$, $\bfZ_1$, $\bfZ_2$, and $\bfZ = \bfZ_1 \otimes \bfZ_2$ are invertible, and we may set $\bfZ_1^\top = \bfY_1^{-1}$ and $\bfZ_2^\top = \bfY_2^{-1}$. 
Let $\tilde{\bfU}_1$ and $\tilde{\bfU}_2$ be the first $n$ columns of $\bfU_1$ and $\bfU_2$, respectively, and let $\tilde{\bfV}_1$ and $\tilde{\bfV}_2$ be the last $n$ columns of $\bfV_1$ and $\bfV_2$, respectively.
Define also $\tilde{\bfB} = \tbfU_2^\top\bfB\tbfU_1$ and $\tilde{\bfH}^{(k)} = \tilde{\bfV}_2^\top\bfH^{(k)}\tilde{\bfV}_1$, where $\bfB = \text{array}(\bfb)$ and $\bfH^{(k)} = \text{array}(\bfh^{(k)})$.  Then, using \cref{eq:GSVDKP} and \begin{align}\label{eq:KPvec}
    (\bfC \otimes \bfD)\text{vec}(\bfF) = \text{vec}(\bfD\bfF\bfC^\top),
\end{align}
where $\text{vec}(\bfF)$ is the vectorization of the matrix $\bfF$, the solution in \cref{eq:innersol} can be written as
\begin{align} \label{eq:xsolh}
    \bfx = \text{vec}(\bfY_2[ \tbfPhi \odot \tilde{\bfB}+ \tbfPsi \odot \tilde{\bfH}^{(k)})]\bfY_1^\top),
\end{align}
where $\odot$ denotes element-wise multiplication and $\tbfPhi$ and $\tbfPsi$ are given by
\begin{align*}
    \tbfPhi &= \text{array}\left(\left[\frac{\upsilon_1}{\upsilon_1^2 +\lambda^2\mu_1^2},\dots,\frac{\upsilon_{n^2}}{\upsilon_{n^2}^2 +\lambda^2\mu_{n^2}^2}\right]^\top\right) \\
    \tbfPsi &= \text{array}\left(\left[\frac{\lambda^2\mu_1}{\upsilon_1^2 +\lambda^2\mu_1^2},\dots,\frac{\lambda^2\mu_{n^2}}{\upsilon_{n^2}^2 +\lambda^2\mu_{n^2}^2}\right]^\top\right).
\end{align*}
Here, $\upsilon_i$ and $\mu_i$ refer to the (potentially) non-zero entry in column $i$ of $\tbfUpsilon_1 \otimes \tbfUpsilon_2$ or $\tbfM_1 \otimes \tbfM_2$.

In both SB and MM, $\bfL\bfx$ must also be computed at each iteration.  
This computation can be written using \cref{eq:KPvec} as 
\begin{align}\label{eq:Lx}
    \bfL\bfx = \text{vec}\left(\tilde{\bfV}_2
       [\hat{\bfM} \odot(\bfY_2^{-1}\bfX\bfY_1^{-\top})]\tilde{\bfV}_1^\top\right),
\end{align}
where $\hat{\bfM} = \text{array}(\left[\mu_1,\dots,\mu_{n^2}\right]^\top)$. 

In the $\chi^2$ dof test, we also compute $\bfL^\dagger_\bfA \bfh^{(k)}$.  Defining $\overline{\bfm} \in \mathbb{R}^{n^2}$ by 
\begin{align*}
    \overline{m}_i = \begin{cases}
        1/\mu_i, &\mu_i >0 \\
        0, &\mu_i = 0
    \end{cases},
\end{align*}
and $\overline{\bfM} = \text{array}(\overline{\bfm})$, we have that
\begin{equation}\label{eq:LAh}
    \bfL^\dagger_\bfA \bfh^{(k)} = \text{vec}(\bfY_2\overline{\bfM}\tilde{\bfH}^{(k)}\bfY_1^\top).
\end{equation}
Notice that \cref{eq:xsolh}, \cref{eq:Lx}, and \cref{eq:LAh} allow each step in the iterative methods to be computed using only the individual GSVDs in \cref{eq:2GSVDs}.

\subsection{Theoretical Properties of the GSVD of Column-Orthogonal Matrices}
Here, we examine the properties of the GSVD of a matrix pair $\{\bfA,\bfL\}$ where one is a column-orthogonal matrix. We refer to a matrix $\bfC \in \mathbb{R}^{t \times n}$, $t \geq n$, as column orthogonal if $\bfC^\top\bfC = \bfI_n$. 
First, we recall the singular value decomposition (SVD) of $\bfA$. The SVD is a decomposition of the form $\bfA = \bfU_\bfA\boldsymbol{\Sigma}\bfV_\bfA^\top$, where $\bfU_\bfA \in \mathbb{R}^{m \times m}$ and $\bfV_\bfA \in \mathbb{R}^{n \times n}$ are orthogonal and $\boldsymbol{\Sigma} \in \mathbb{R}^{m \times n}$ is diagonal with entries $\sigma_1,\dots,\sigma_{\tilde{m}}$ and $\tilde{m} = \min\{m,n\}$. These diagonal entries are known as the singular values (SVs) of $\bfA$ and are sorted in decreasing order: $\sigma_1 \geq \sigma_2 \geq \dots \geq \sigma_{\tilde{m}} \geq 0$.  The rank of $\bfA$ is equal to the number of non-zero SVs.  
It is well known that when $\bfL = \bfI$, the GSVD of $\{\bfA,\bfL\}$ reduces to the SVD of $\bfA$ with $\gamma_i = \sigma_i$, where the $\gamma_i$ are sorted in decreasing order \cite{schimpf1997robust,van1989analysis}.  If $\bfL$ has full column rank, the GSVs are the SVs of $\bfA\bfL^\dagger$ \cite{dykes2014simplified,edelman2020gsvd,li2024characterizing,zha1996computing}. 
If, however, $\bfL$ is column orthogonal, we can say more about the GSVD of $\{\bfA,\bfL\}$ as described by the following lemma. 
\begin{lemma} \label{lem:orthoGSV}
    Let $\bfL$ be column orthogonal and $k = \text{rank}(\bfA)$.  Then, in the GSVD of $\{\bfA,\bfL\}$,   $\gamma_i = \sigma_i$ for $i=1,\dots, \tilde{m}$ and $\gamma_i = 0$ for $i=\tilde{m}+1,\dots,n$. Further, $\bfY = \bfV_\bfA \bfG$, where $\bfG = \text{diag}(1/\sqrt{\sigma_1^2+1},\dots,1/\sqrt{\sigma_k^2+1},\bfone_{n-k})$. The $\gamma_i$ are sorted in decreasing order as in \cref{lem:shortgsvd}.
\end{lemma}
\begin{proof}
Let $k$ be the rank of $\bfA$ and 
let the SVD of $\bfA$ be given by $\bfA = \bfU_\bfA\boldsymbol{\Sigma}\bfV_\bfA^\top$. Since $\bfV_\bfA$ is orthogonal, we can write that $\bfL = \bfL \bfV_\bfA\bfI\bfV_\bfA^\top$.  This provides a joint decomposition of the form
 \begin{align}
\begin{bmatrix} \bfA \\ \bfL \end{bmatrix} = \begin{bmatrix}
\bfU_\bfA & {\bf{0}} \\ {\bf{0}} & \bfL \bfV_\bfA \end{bmatrix} \begin{bmatrix}
\boldsymbol{\Sigma} \\ \bfI \end{bmatrix}
\bfV_\bfA^\top.
\end{align}
Since $\bfG$ is invertible, we have that
\begin{align} \label{eq:GSVDSV}
\begin{bmatrix}
\boldsymbol{\Sigma} \\ \bfI \end{bmatrix}
\bfV_\bfA^\top =  \begin{bmatrix}
\boldsymbol{\Sigma} \\ \bfI \end{bmatrix} \left(\bfG \bfG^{-1}\right)
\bfV_\bfA^\top  =
\left( \begin{bmatrix}
\boldsymbol{\Sigma} \\ \bfI \end{bmatrix} \bfG\right) \left(\bfG^{-1}
\bfV_\bfA^\top\right) =
\begin{bmatrix}
\boldsymbol{\Upsilon} \\ \bfM \end{bmatrix} \bfY^{-1},
 \end{align}  
 where $\bfUpsilon = \boldsymbol{\Sigma} \bfG$, $\bfM = \bfG$, and $\bfY = \bfV_\bfA \bfG$.  With this, 
 \begin{align*}
     \bfUpsilon = \text{diag}\left(\frac{\sigma_1}{\sqrt{\sigma_1^2+1}},\dots,\frac{\sigma_k}{\sqrt{\sigma_k^2+1}},{\bf{0}}_{\tilde{m}-k}\right) \text{and } \bfM = \text{diag}\left(\frac{1}{\sqrt{\sigma_1^2+1}},\dots,\frac{1}{\sqrt{\sigma_k^2+1}},{\bf{1}}_{n-k}\right).
 \end{align*}
Hence, \cref{eq:GSVprop} is satisfied.  $\bfU_\bfA$ is orthogonal from the SVD and $\bfL\bfV_\bfA$ is column orthogonal since $\bfL$ is column orthogonal and $\bfV_\bfA$ is orthogonal. $\bfL\bfV_\bfA$ is not necessarily orthogonal since $\bfL$ may have more rows than columns. In this case, we add columns to the left of $\bfL\bfV_\bfA$ to make it orthogonal. To maintain the decomposition, we also add rows of zeros to the top of $\bfM$ that correspond to these additional columns in $\bfL\bfV_\bfA$.  Following this, the decomposition is then a GSVD.

Notice that the GSVs for $i=1,\dots,k$ are
\begin{align*}
    \gamma_i = \frac{\upsilon_i}{\mu_i} = \frac{\sigma_i/\sqrt{\sigma_i^2+1}}{1/\sqrt{\sigma_i^2+1}} = \sigma_i,
\end{align*}
and for $i=k+1,\dots,\tilde{m}$ are $\gamma_i = 0/1 = 0 = \sigma_i$.
Thus, if $m \geq n$, the GSVs are the SVs of $\bfA$.  If $m <n$, then the GSVs for $i=m+1,\dots,n$ come from the columns of zeros in $\bfUpsilon$ and are $\gamma_i = 0/1 = 0$.   
From \cref{eq:GSVDSV}, we can also see that $\bfY$ is $\bfV_\bfA$ from the SVD of $\bfA$ with scaled columns.  In particular, the scaling implies that for $i=1,\dots,k$, $\mu_i = \norm{y_i}_2$ and $\upsilon_i = \sigma_i \norm{y_i}_2$. 
This also shows that $\bfV$ is related to $\bfV_\bfA$ and, by extension, to $\bfY$, since $\bfV=\bfL\bfV_\bfA = \bfL \bfY \bfG^{-1}$.
\end{proof}

If $\bfA$ is column orthogonal instead of $\bfL$, we obtain a similar result by flipping the roles of $\bfA$ and $\bfL$ and writing a decomposition of $\bfA$ in terms of the SVD of $\bfL$. In this case, the SVD of $\bfL$ determines the matrices in the GSVD of $\{\bfA,\bfL\}$. 
On the other hand, if $\bfA$ is column orthogonal, then $\bfA^\top\bfA=\bfI$ and the system of equations is not ill-posed. Thus, we immediately obtain $\bfx=\bfA^\top \bfb$, and we do not need the GSVD.

To obtain the same result as in \cref{lem:orthoGSV}, we can also consider the generalized eigenvalue problem \cite{parlett1998symmetric} involving the matrix pair $\{\bfA^\top\bfA,\bfL^\top\bfL\}$, in which the goal is to find a scalar $\lambda_g$ and non-zero vector $\bfx_g$ such that
\begin{align} \label{eq:geneig2}
    \bfA^\top\bfA \bfx_g = \lambda_g \bfL^\top\bfL \bfx_g,
\end{align}
where $\bfA^\top\bfA,\bfL^\top\bfL \in \mathbb{R}^{n \times n}$ are symmetric.
When $\bfL$ is column orthogonal, \cref{eq:geneig2} simplifies to the eigenvalue problem
\begin{align} \label{eq:eig2}
    \bfA^\top\bfA \bfx_s = \lambda_s \bfx_s,
\end{align}
since $\bfL^\top\bfL = \bfI$ \cite{sriperumbudur2011majorization,wang2016recurrent}.
\begin{lemma}
    When $\bfL$ is column orthogonal, the $\lambda_g$ and $\bfx_g$ that solve \cref{eq:geneig2} are the same as the $\lambda_s$ and $\bfx_s$ that solve \cref{eq:eig2}: $\lambda_g = \lambda_s$ and $\bfx_g = \bfx_s$. 
\end{lemma}
\begin{proof}
Since $\bfA^\top\bfA + \bfL^\top\bfL$ is positive definite, there exists a solution to \cref{eq:geneig2} and any $\lambda_g$ and $\bfx_g$ that solve \cref{eq:geneig2} are real~\cite{gao2008continuous,parlett1998symmetric}. 
The $\lambda_g$ is a generalized eigenvalue and $\bfx_g$ is the associated eigenvector.  We can write \cref{eq:geneig2} to include all eigenvalues and eigenvectors as a matrix equation
\begin{align} \label{eq:geneig}
    \bfA^\top\bfA \bfPhi = \bfL^\top\bfL \bfPhi \boldsymbol{\Lambda},
\end{align}
where $\bfPhi \in \mathbb{R}^{n \times n}$ contains the $n$ generalized eigenvectors and $\boldsymbol{\Lambda} \in \mathbb{R}^{n \times n}$ is a diagonal matrix with the generalized eigenvalues.
When $\bfL$ is column orthogonal, \cref{eq:geneig2} simplifies to the eigenvalue problem for $\bfA^\top\bfA$:
\begin{align*} \label{eq:eig}
    \bfA^\top\bfA \bfPhi &= \bfPhi \boldsymbol{\Lambda}\\
    \bfPhi^\top \bfA^\top\bfA \bfPhi &=  \boldsymbol{\Lambda}.
\end{align*}
Thus, the eigenvalues of $\bfA^\top\bfA$ are the eigenvalues of $\boldsymbol{\Lambda}$, meaning that the generalized eigenvalues of $\{\bfA^\top\bfA,\bfL^\top\bfL\}$ are the eigenvalues of $\bfA^\top\bfA$.  The associated eigenvectors for $\bfA^\top\bfA$ are the columns of $\bfPhi$, which are also the generalized eigenvectors of $\{\bfA^\top\bfA,\bfL^\top\bfL\}$ from \cref{eq:geneig}. When $m \geq n$, this analysis can  be connected back to the GSVs. Since the SVs of $\bfA$ are the square roots of the eigenvalues of $\bfA^\top\bfA$ and the finite GSVs of $\{\bfA,\bfL\}$ are the square roots of the finite generalized eigenvalues of $\{\bfA^\top\bfA,\bfL^\top\bfL\}$ \cite{betcke2008generalized}, the GSVs of $\{\bfA,\bfL\}$ are the SVs of $\bfA$.
\end{proof}

From the proof of \cref{lem:orthoGSV}, we can see that when $\bfL$ is column orthogonal, the GSVD of $\{\bfA,\bfL\}$ can be obtained from the SVD of $\bfA$. This means that instead of computing the GSVD, we can use the SVD of $\bfA$.
Another result of \cref{lem:orthoGSV} is that the basis for the solution of \cref{eq:Lxh2} is 
determined solely by the right singular vectors of $\bfA$ and not by $\bfL$.  Similarly, the GSVs are also determined by the SVs of $\bfA$. 
The only matrix in the GSVD that changes between the framelets and wavelets is the $\bfV$ matrix, which is used in the solution \cref{eq:xsolh} when $\bfh^{(k)} \neq {\bf{0}}$. When $\bfh^{(k)} = {\bf{0}}$ though, as in the generalized Tikhonov problem \cref{eq:GenTik}, both $\bfL$ will produce the same solution, leading to the following corollary.

\begin{corollary} \label{cor:orthocol}
    Let $\bfL$ and $\tilde{\bfL}$ both be column orthogonal. Then, for a given $\lambda$, it makes no difference whether $\bfL$ or $\tilde{\bfL}$ is used as the regularization matrix for the generalized Tikhonov problem \cref{eq:GenTik} as the problem reduces to standard Tikhonov with $\bfL = \bfI$ for both matrices. Thus, the solutions will be the same, to machine precision. 
\end{corollary}

When $\bfL$ is column orthogonal, \cref{eq:xsolh} can be written as
\begin{align} \label{eq:xsolh2}
    \bfx = \text{vec}(\bfY_2[ \overline{\tbfPhi} \odot \tilde{\bfB}+ \overline{\tbfPsi} \odot \tilde{\bfH}^{(k)})]\bfY_1^\top),
\end{align}
where
\begin{align*}
    \overline{\tbfPhi} = \text{array}\left(\left[\frac{\upsilon_1}{\upsilon_1^2 +\lambda^2},\dots,\frac{\upsilon_{n^2}}{\upsilon_{n^2}^2 +\lambda^2}\right]^\top\right) \mbox{ and }
    \overline{\tbfPsi} = \text{array}\left(\left[\frac{\lambda^2\mu_1}{\upsilon_1^2 +\lambda^2},\dots,\frac{\lambda^2\mu_{n^2}}{\upsilon_{n^2}^2 +\lambda^2}\right]^\top\right).
\end{align*}

\section{Family of Regularization Operators}\label{sec:regop}

This section presents a family of regularization operators $\bfL$, which we adopted for the solution of \cref{eq:Lxh}. Specifically, we will consider a two-level framelet analysis operator \cite{daubechies2003framelets,ron1997affine} and the Daubechies D4 discrete wavelet transform \cite{daubechies1992ten}.  

Consider a matrix that corresponds to a one-dimensional (1D) transform of the form \begin{align*}
    \bfL_{T}^{1D} = \begin{bmatrix}
        \bfT_0 \\ \bfT_1 \\ \vdots\\ \bfT_t
    \end{bmatrix},
\end{align*}
and its two-dimensional (2D) transform consists of the form \begin{align*}
    \bfL_T^{2D} = \begin{bmatrix}
        \bfT_{00} \\ \bfT_{01} \\ \vdots \\ \bfT_{ij} \\ \vdots \\ \bfT_{tt}
    \end{bmatrix},
\end{align*}
where $\bfT_{ij} = \bfT_i \otimes \bfT_j$ for $i,j=1,\dots,t.$
This then yields the  2D-transform-regularized problem
\begin{align} \label{eq:direct2D}
    \min_\bfx\left\{\frac{1}{2}\norm{(\bfA_1\otimes \bfA_2)\bfx-\bfb}_2^2 + \frac{\mu}{2}\norm{\bfL_T^{2D}\bfx}_1\right\}.
\end{align}

To solve \cref{eq:direct2D} we observe that
$\bfL_{T}^{1D} \otimes \bfL_{T}^{1D} = \bfP_T \bfL_T^{2D}$, 
where $\bfP_T$ is a permutation matrix.  Since 
\begin{align*} \label{eq:perm1norm}
     \norm{\bfL_T^{2D}\bfx}_1 = \norm{\bfP\bfL_T^{2D}\bfx}_1
\end{align*} 
for any permutation matrix $\bfP$, we have that
\begin{align*}
     \norm{\bfL_T^{2D}\bfx}_1 =\norm{\bfP_T\bfL_T^{2D}\bfx}_1= \norm{(\bfL_{T}^{1D} \otimes \bfL_{T}^{1D})\bfx}_1.
\end{align*} 
As a result, solving
\begin{align}\label{eq:frameletreg}
    \min_\bfx\left\{\frac{1}{2}\norm{(\bfA_1\otimes\bfA_2)\bfx-\bfb}_2^2 + \frac{\mu}{2}\norm{(\bfL_{T}^{1D} \otimes \bfL_{T}^{1D})\bfx}_1\right\},
\end{align}
is equivalent to solving \cref{eq:direct2D}.  Thus, we will solve \cref{eq:frameletreg} as it provides the advantage of the KP structure. 

\subsection{Framelets}\label{subsec:frames}
Following the presentation in \cite{buccini2020modulus}, we use tight frames that are determined by B-splines for regularization.  With these frames, we have one low-pass filter $\bfF_0 \in \mathbb{R}^{n \times n}$ and two high pass filters $\bfF_1,\bfF_2 \in \mathbb{R}^{n \times n}$ given by
\begin{equation*}
\scriptsize
    \bfF_0 = \frac{1}{4}\begin{bmatrix}
    3 & 1 & 0 & \cdots & 0 \\
    1 & 2 & 1 & & \\
    & \ddots & \ddots & \ddots & \\
    && 1 & 2 &1 \\
    0 & \cdots & 0 & 1 & 3
    \end{bmatrix}, \quad \bfF_1 = \frac{\sqrt{2}}{4}\begin{bmatrix}
    -1 & 1 & 0 & \cdots & 0 \\
    -1 & 0 & 1 & & \\
    & \ddots & \ddots & \ddots & \\
    && -1 & 0 &1 \\
    0 & \cdots & 0 & -1 & 1
    \end{bmatrix} \text{ and }
    \bfF_2 = \frac{1}{4}\begin{bmatrix}
    1 & -1 & 0 & \cdots & 0 \\
    -1 & 2 & -1 & & \\
    & \ddots & \ddots & \ddots & \\
    && -1 & 2 &-1 \\
    0 & \cdots & 0 & -1 & 1
    \end{bmatrix}.
\end{equation*}
The one-dimensional operator $\bfL_{F}^{1D} \in  \mathbb{R}^{3n \times n}$ is given by
\begin{align*}
    \bfL_{F}^{1D} = \begin{bmatrix}
        \bfF_0 \\ \bfF_1 \\ \bfF_2
    \end{bmatrix}.
\end{align*}
Since we are concerned with image deblurring in two dimensions, the two-dimensional operator is then constructed using the Kronecker product:
\begin{align*}
    \bfF_{ij} = \bfF_i \otimes \bfF_j, \quad i,j = 0,1,2.
\end{align*}
The matrix $\bfF_{00}$ is a low-pass filter, while each of the other matrices contains at least one of the high-pass filters. The 2D operator  $\bfL_F^{2D} \in \mathbb{R}^{9n^2 \times n^2}$,
obtained by the concatenation of these matrices, is given by  
\begin{align*}
    \bfL_F^{2D} = \begin{bmatrix}
        \bfF_{00} \\ \bfF_{01} \\ \bfF_{02} \\ \vdots \\ \bfF_{22}
    \end{bmatrix}.
\end{align*}

\subsection{Wavelets}\label{subsec:wavelets} 
Discrete wavelet transforms are composed of a low-pass filter, $\bfW_1^\top$ and a high-pass filter, $\bfW_2^\top$ \cite{espanol2009multilevel,van2019discrete}.  For a discrete wavelet transform of order $k$, let $h_0, h_1,\dots, h_{k-1}$ be the filter coefficients.  Then, the filter matrices have the forms 
\begin{align*}
    \bfW_1^\top &= \begin{bmatrix}
        h_0 & h_1 & \cdots & \cdots & h_{k-2} & h_{k-1} & & & & & & \\
        && h_0 & h_1 & \cdots & \cdots & h_{k-2} & h_{k-1} &&&& \\
        &&&& \cdots & \cdots &\cdots &\cdots&\cdots &\cdots && \\
        &&&&&& h_0 & h_1 & \cdots & \cdots & h_{k-2} & h_{k-1} \\
        h_{k-2} & h_{k-1} &&&&&&& h_0 & h_1& \cdots & \cdots \\
        \cdots &\cdots &\cdots&\cdots &&&&&&&\cdots & \cdots \\
        \cdots & \cdots & h_{k-2} & h_{k-1}&&&&&&&h_0&h_1
    \end{bmatrix}
    \\
    \bfW_2^\top &= \begin{bmatrix}
        g_0 & g_1 & \cdots & \cdots & g_{k-2} & g_{k-1} & & & & & & \\
        && g_0 & g_1 & \cdots & \cdots & g_{k-2} & g_{k-1} &&&& \\
        &&&& \cdots & \cdots &\cdots &\cdots&\cdots &\cdots && \\
        &&&&&& g_0 & g_1 & \cdots & \cdots & g_{k-2} & g_{k-1} \\
        g_{k-2} & g_{k-1} &&&&&&& g_0 & g_1& \cdots & \cdots \\
        \cdots &\cdots &\cdots&\cdots &&&&&&&\cdots & \cdots \\
        \cdots & \cdots & g_{k-2} & g_{k-1}&&&&&&&g_0&g_1
    \end{bmatrix},
\end{align*}
where $g_j =(-1)^j h_{k-j-1}, \; j=0,\dots,k-1$. The 1D DWT matrix is then 
\begin{align*}
    \bfL_{W}^{1D} = \begin{bmatrix}
        \bfW^\top_1 \\ \bfW_2^\top
    \end{bmatrix}.
\end{align*}
As with the framelets, the 2D operator is constructed by taking the Kronecker products of the low-pass and high-pass filters:
\begin{align*}
    \bfW^\top_{ij} = \bfW^\top_i \otimes \bfW^\top_j, \quad i,j = 0,1,2.
\end{align*}
The 2D operator is then
\begin{align*}
    \bfL_W^{2D} = \begin{bmatrix}
        \bfW_{11} \\ \bfW_{12} \\ \bfW_{21} \\ \bfW_{22}
    \end{bmatrix}.
\end{align*}
For the examples, we will use the D4 wavelet transform, which is of order 4 and has the filter coefficients
\begin{align*}
    h_0 &= \frac{1+\sqrt{3}}{2}, \, h_1 = \frac{3+\sqrt{3}}{2},\, h_2 = \frac{3-\sqrt{3}}{2}, \,h_3 = \frac{1-\sqrt{3}}{2}. 
\end{align*}

\section{Numerical Results}\label{sec:num}
In this section, the SB and MM methods are applied to image deblurring examples where the blur is separable. For the blurring matrix $\bfA$, we consider $\bfA_j  \in \mathbb{R}^{n \times n}$, $j=1$, $2$, to be a one-dimensional blur with either zero or periodic boundary conditions.  With zero boundary conditions, $\bfA_j$ is a symmetric Toeplitz matrix with the first row
\begin{align*}
    \bfz_j = \frac{1}{\sqrt{2\pi}\tilde{\sigma}_j}\begin{bmatrix}
        \text{exp}\left(\frac{-[0:\text{band}_j-1]^2}{2\tilde{\sigma}_j^2}\right) & {\bfzero_{n-\text{band}_j}}
    \end{bmatrix}, \label{eq:zrow1}
\end{align*}
while with periodic boundary conditions, $\bfA_j$ is a circulant matrix with first row
\begin{align*}
    \bfz_j = \frac{1}{\sqrt{2\pi}\tilde{\sigma}_j}\begin{bmatrix}
        \text{exp}\left(\frac{-[0:\text{band}_j-1]^2}{2\tilde{\sigma}_j^2}\right) & {\bfzero_{n-2\text{band}_j+1}} & \text{exp}\left(\frac{-[\text{band}_j-1:1]^2}{2\tilde{\sigma}_j^2}\right)
    \end{bmatrix}. \label{eq:zrow2}
\end{align*}
Here, $\tilde{\sigma}_j$ and $\text{band}_j$ are blur parameters.
We then add Gaussian noise $\bfe$ to $\bfb_{true}$ according to a specified blurred signal to noise ratio (BSNR), which is defined as
\begin{align*}\label{eq:BSNR}
    \text{BSNR} = 20\log_{10}\left(\frac{\norm{\bfb_{true}}_2}{\norm{\bfe}_2}\right).
\end{align*}   

The image is then deblurred using the SB and MM methods, with $\bfL$ being defined by framelets or wavelets as described in \cref{sec:regop}. We set set $tol=0.01$ and $K_{max} = 20$.   The methods are compared using the relative error (RE) defined by
\begin{align*}\label{eq:RE}
    \text{RE} = \frac{\norm{\bfx-\bfx_{true}}_2}{\norm{\bfx_{true}}_2},
\end{align*}
and the improved signal to noise ratio (ISNR), which is defined by
\begin{align*}\label{eq:ISNR}
    \text{ISNR} = 20\log_{10}\left(\frac{\norm{\bfb-\bfx_{true}}_2}{\norm{\bfx-\bfx_{true}}_2}\right).
\end{align*}
Both metrics measure the quality of the solution, with a smaller $\text{RE}$ and a larger $\text{ISNR}$ indicating a better solution.  

\subsection{Blurring Examples}\label{subsec:blur}

\begin{figure} 
    \centering
    \subfigure[$\bfx$ \label{Ex1:x}]{
\includegraphics[width=3.5cm]{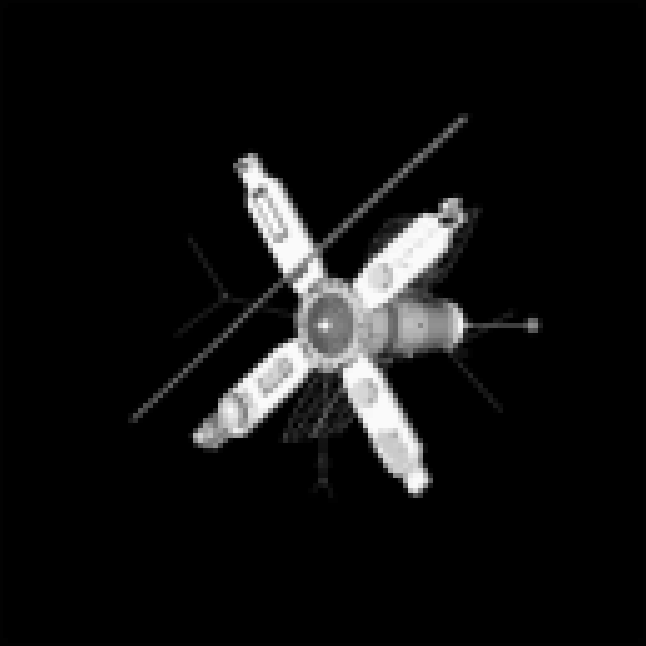}}
    \subfigure[PSF \label{Ex1:PSF}]{
\includegraphics[width=3.5cm]{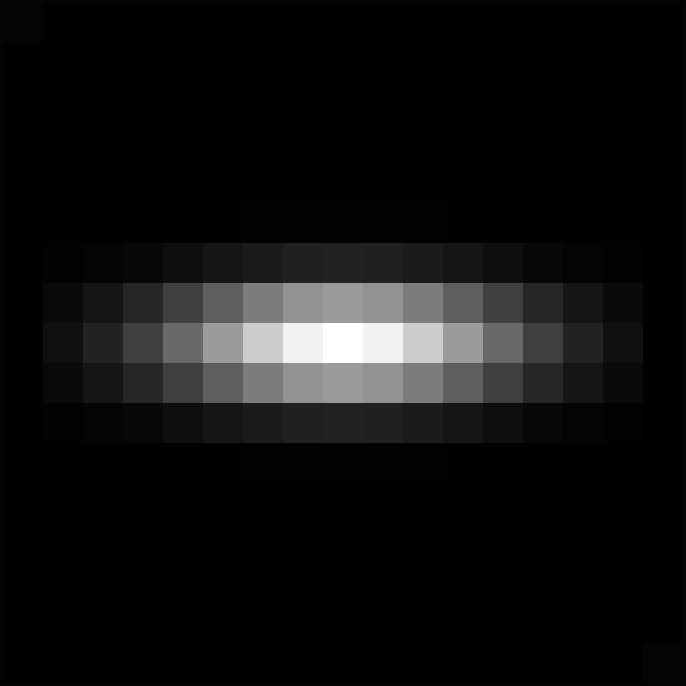}}
    \subfigure[$\bfb_{true}$ \label{Ex1:btrue}]{
\includegraphics[width=3.5cm]{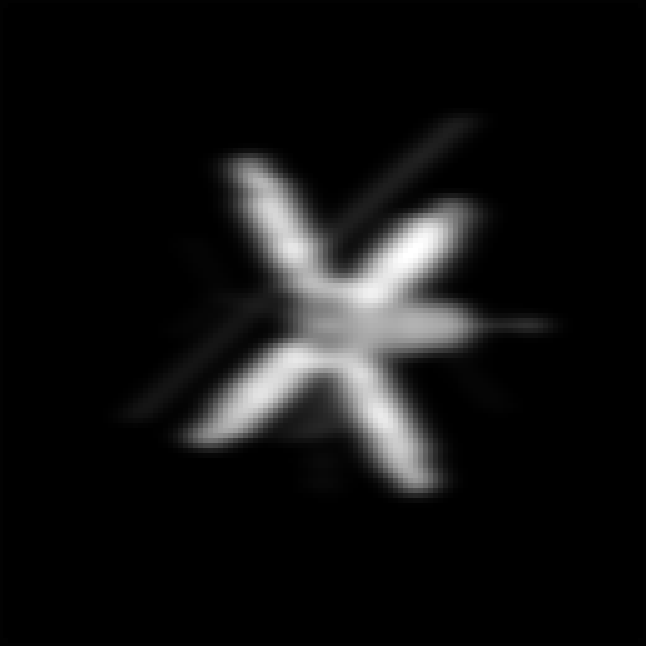}}
    \subfigure[$\bfb$ \label{Ex1:b}]{
\includegraphics[width=3.5cm]{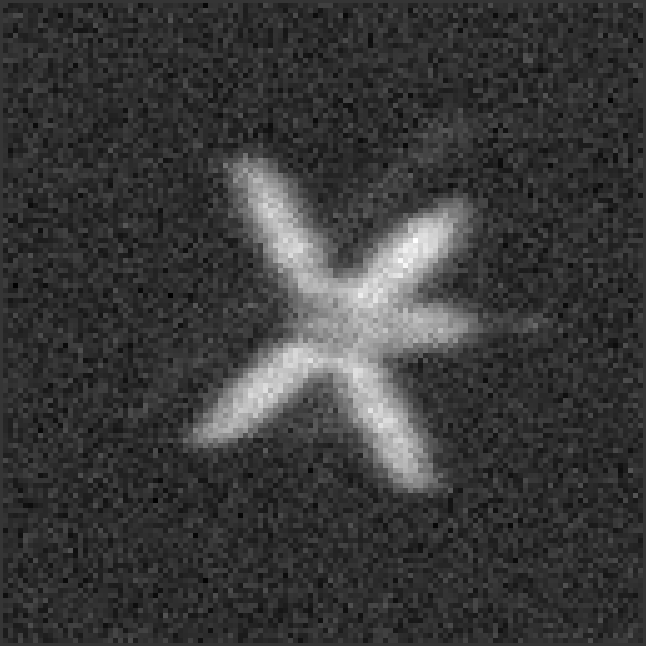}}
\caption{True image $\bfx$ ($128 \times 128$ pixels), PSF ($17 \times 17$ pixels), blurred image $\bfb_{true}$, and the blurred and noisy image $\bfb$ with $\text{BSNR}=10$ for Example 1.}
    \label{Ex1}
\end{figure}

\subsubsection{Example 1}
For the first example, we blur the satellite image in \cref{Ex1:x} with size $n = 128$.  For $\bfA$, we use zero boundary conditions, $\tilde{\sigma}_1 = 3$, $\tilde{\sigma}_2 = 1$, and $\text{band}_1 = \text{band}_2 = 15$. The corresponding point spread function (PSF) is shown in \cref{Ex1:PSF}, and the blurred image $\bfb_{true}$ is shown in \cref{Ex1:btrue}. Gaussian noise with $\text{BSNR} = 10$ is added to obtain $\bfb$ as seen in \cref{Ex1:b}.  This problem is then solved using SB with $\tau=0.04$ and MM with $\varepsilon = 0.03$.
The solutions in \cref{Ex1R} show that framelet regularization produces solutions with fewer artifacts in the background.  The RE and ISNR for these methods are shown in \cref{Ex1Conv}.  This shows that SB and MM perform similarly and that the framelet solutions have larger ISNRs than the wavelet solutions.
The results are summarized in \cref{tab:Ex1} along with timing.  From this, we can observe that with $\chi^2$, wavelet methods converge faster than framelet methods while with GCV, framelet methods converge faster. In all cases, the $\chi^2$ selection method, which is a root-finding technique, is faster than GCV despite GCV taking fewer iterations to converge.  GCV and $\chi^2$ produce solutions with comparable $\text{RE}$ and $\text{ISNR}$ with either framelets or wavelets.

\begin{figure}
    \centering
    \subfigure[SB Framelets, Opt\label{Ex1R:a}]{
\includegraphics[width=3.5cm]{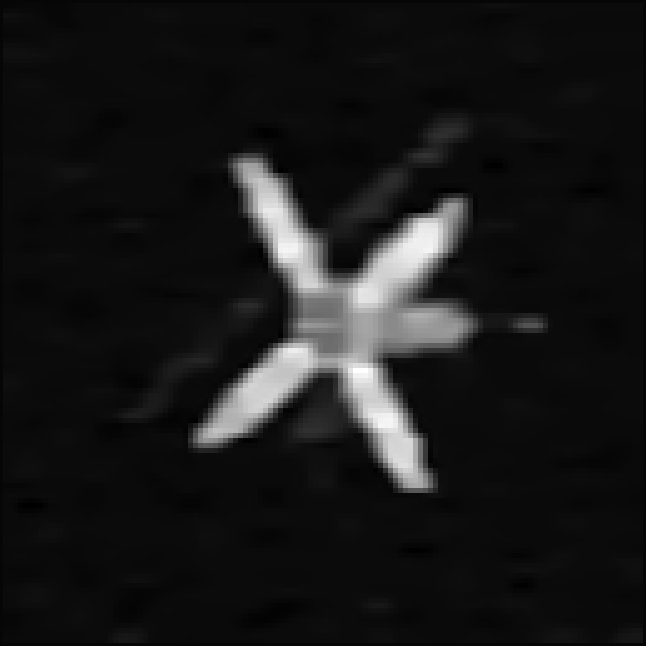}}
    \subfigure[SB Wavelets, Opt\label{Ex1R:b}]{
\includegraphics[width=3.5cm]{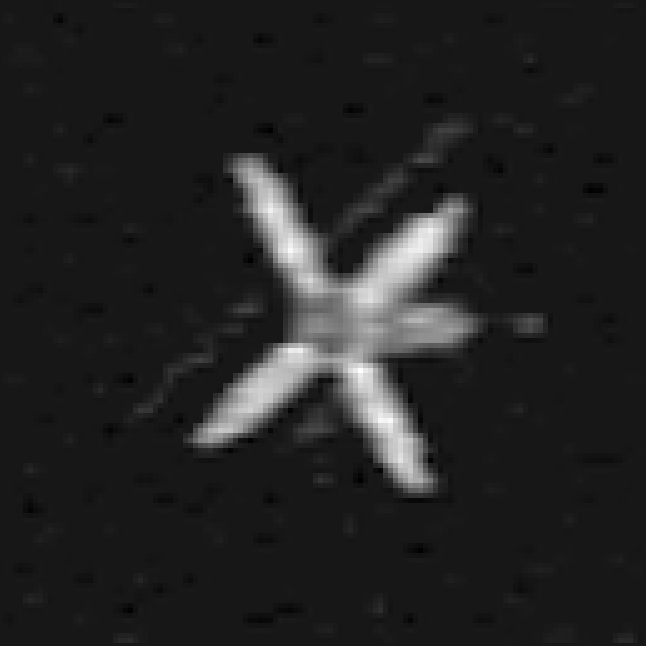}}
    \subfigure[MM Framelets, Opt\label{Ex1R:c}]{
\includegraphics[width=3.5cm]{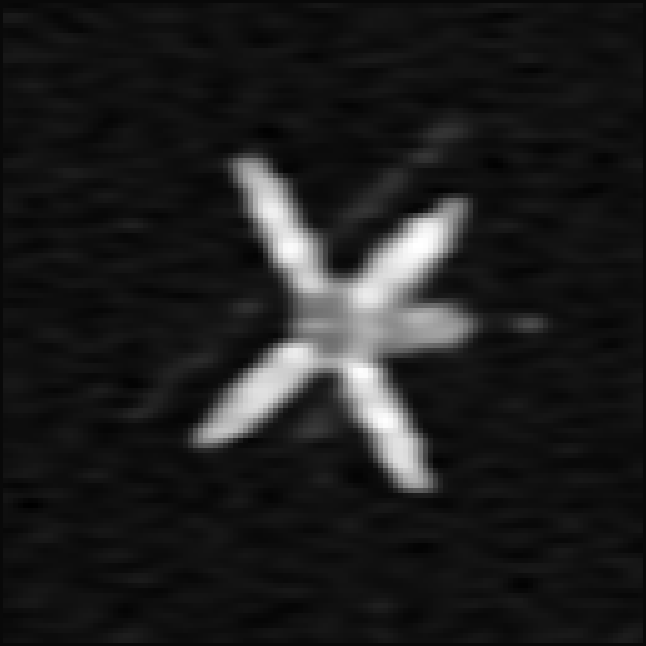}}
    \subfigure[MM Wavelets, Opt\label{Ex1R:d}]{
\includegraphics[width=3.5cm]{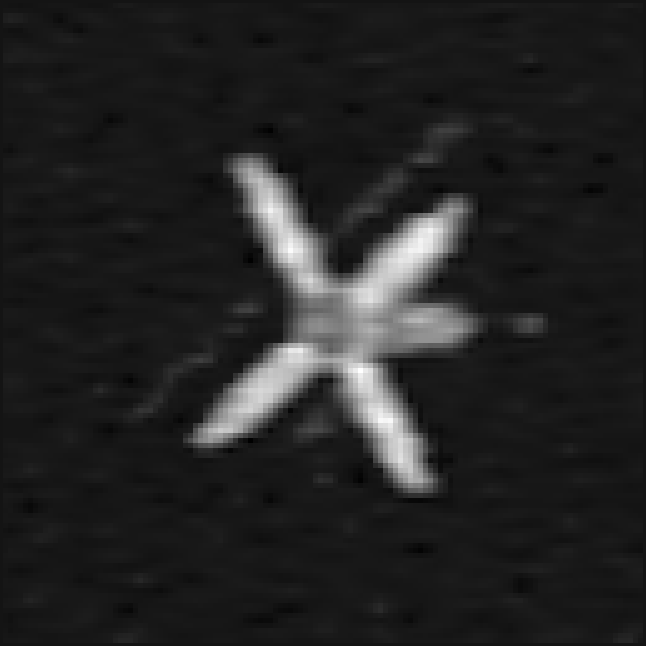}}
\subfigure[SB Framelets, GCV\label{Ex1R:e}]{
\includegraphics[width=3.5cm]{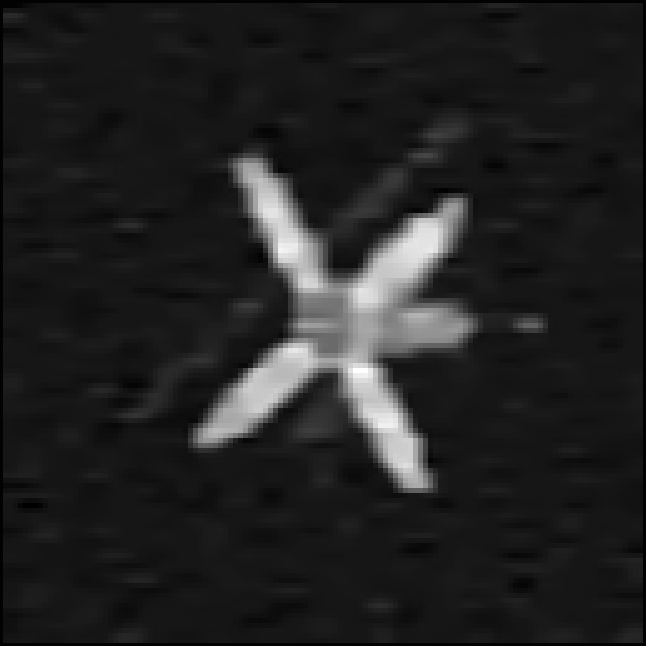}}
    \subfigure[SB Wavelets, GCV \label{Ex1R:f}]{
\includegraphics[width=3.5cm]{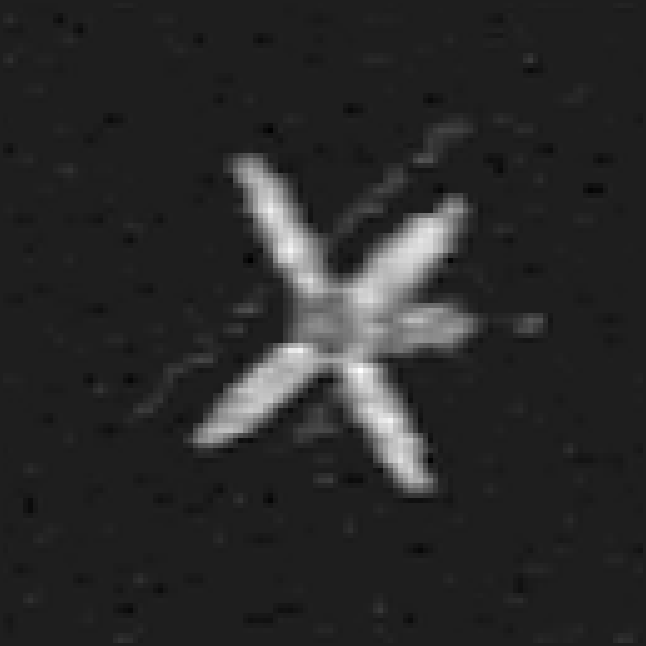}}
\subfigure[MM Framelets, GCV\label{Ex1R:g}]{
\includegraphics[width=3.5cm]{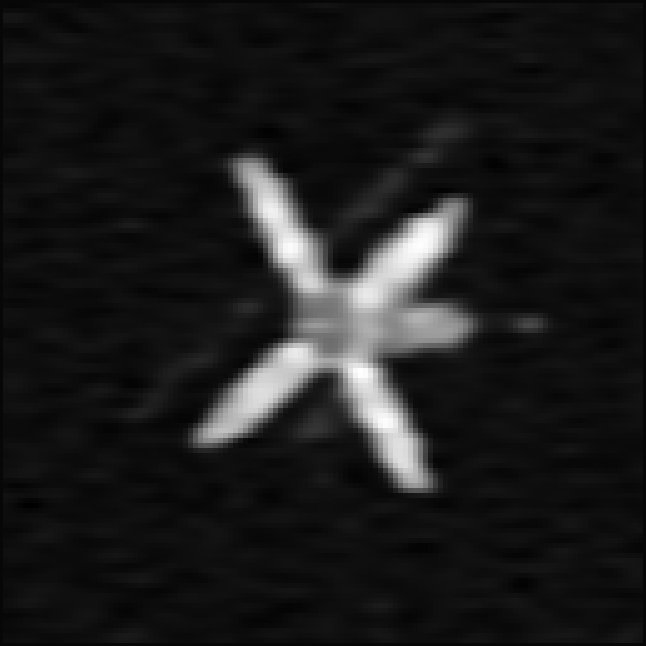}}
    \subfigure[MM Wavelets, GCV \label{Ex1R:h}]{
\includegraphics[width=3.5cm]{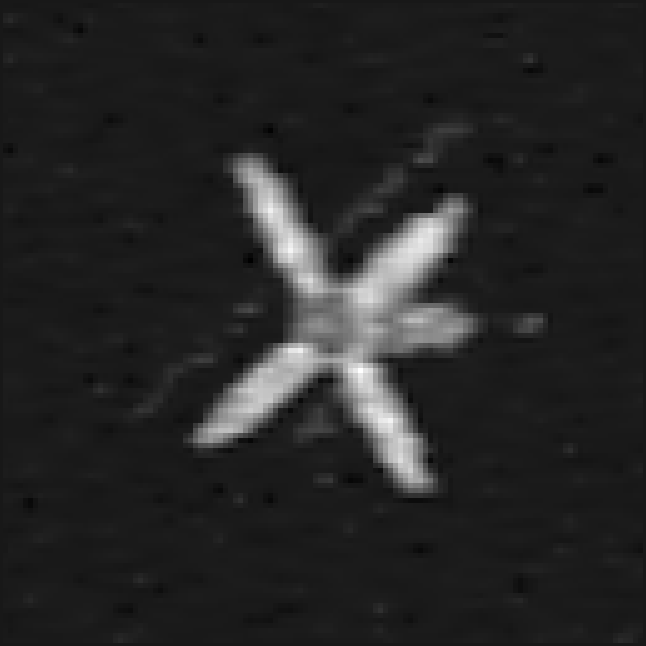}}
   \subfigure[SB Framelets, $\chi^2$ \label{Ex1R:i}]{
\includegraphics[width=3.5cm]{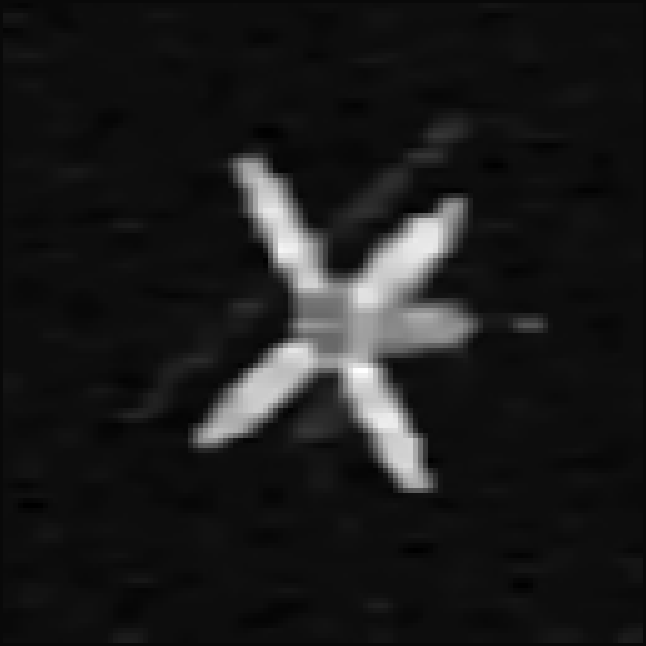}}
    \subfigure[SB Wavelets, $\chi^2$ \label{Ex1R:j}]{
\includegraphics[width=3.5cm]{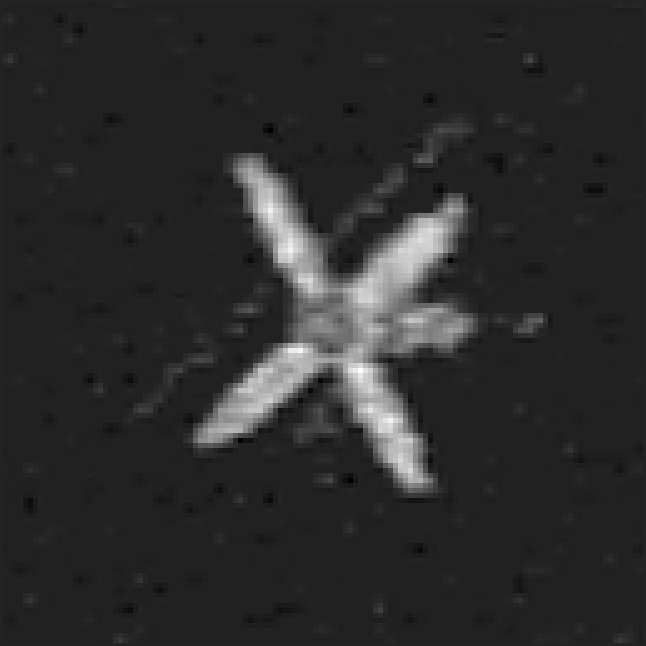}}
\subfigure[MM Framelets, $\chi^2$\label{Ex1R:k}]{
\includegraphics[width=3.5cm]{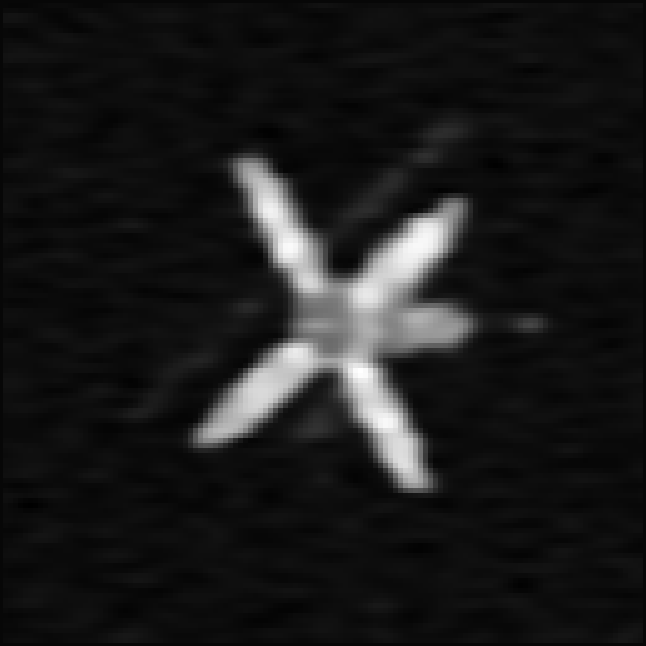}}
    \subfigure[MM Wavelets, $\chi^2$ \label{Ex1R:l}]{
\includegraphics[width=3.5cm]{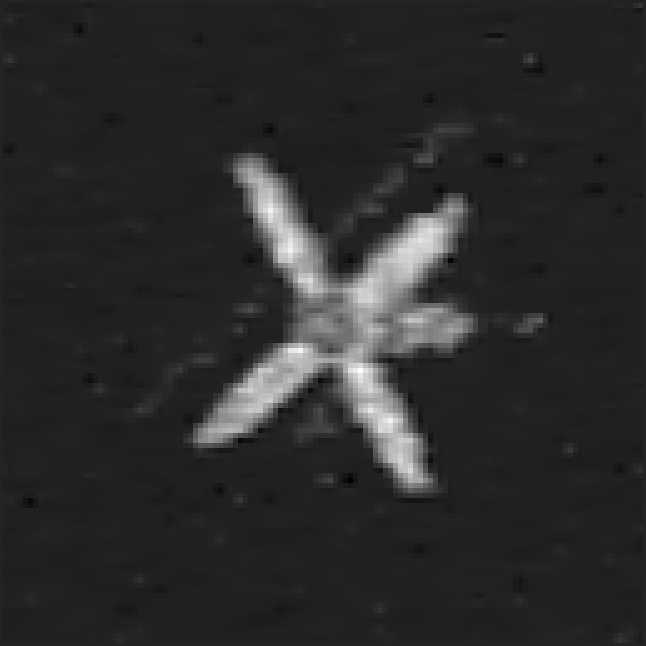}}
    \caption{Solutions to Example 1 in \cref{Ex1}.  The solutions use either SB or MM to solve the problem and either framelets or wavelets for regularization.  The parameter is either fixed optimally or selected with GCV or the non-central $\chi^2$ test.}
    \label{Ex1R}
\end{figure}

\begin{figure} 
    \centering
    \subfigure[$\text{RE}$: SB Framelets \label{Ex1:SBF}]{
\includegraphics[width=3.5cm]{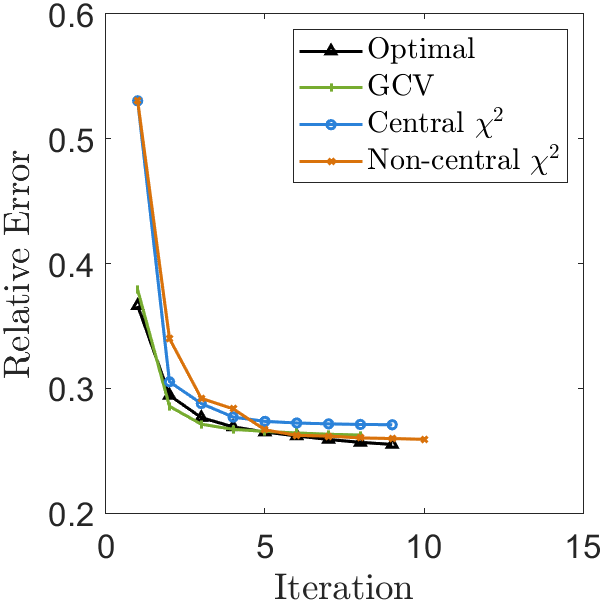}}
    \subfigure[$\text{RE}$: SB Wavelets \label{Ex1:SBW}]{
\includegraphics[width=3.5cm]{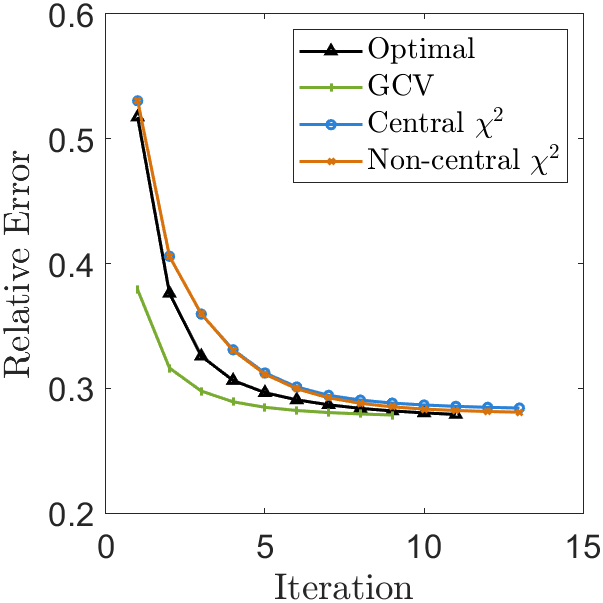}}
   \subfigure[$\text{RE}$: MM Framelets \label{Ex1:MMF}]{
\includegraphics[width=3.5cm]{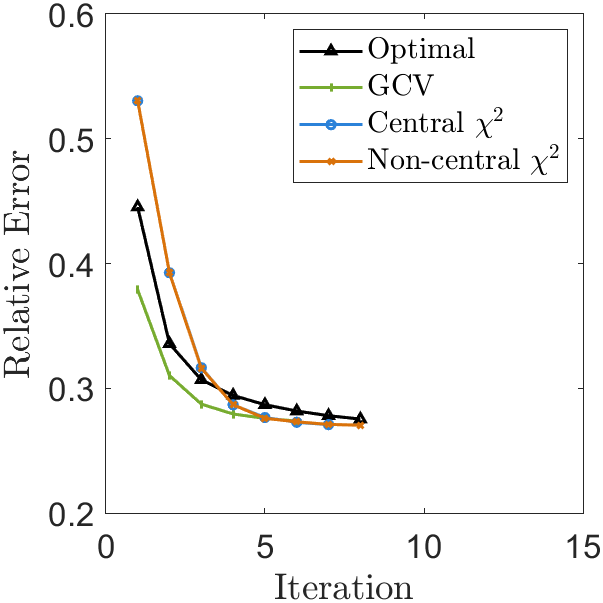}}
    \subfigure[$\text{RE}$: MM Wavelets\label{Ex1:MMW}]{
\includegraphics[width=3.5cm]{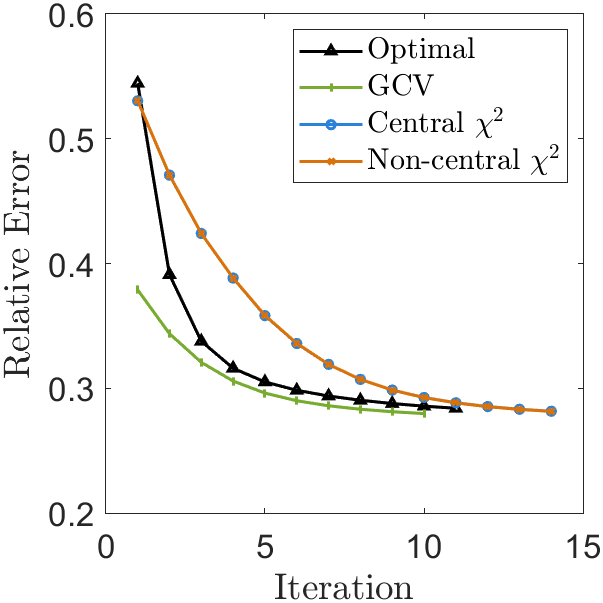}}

    \subfigure[$\text{ISNR}$: SB Framelets \label{Ex1:SBFI}]{
\includegraphics[width=3.5cm]{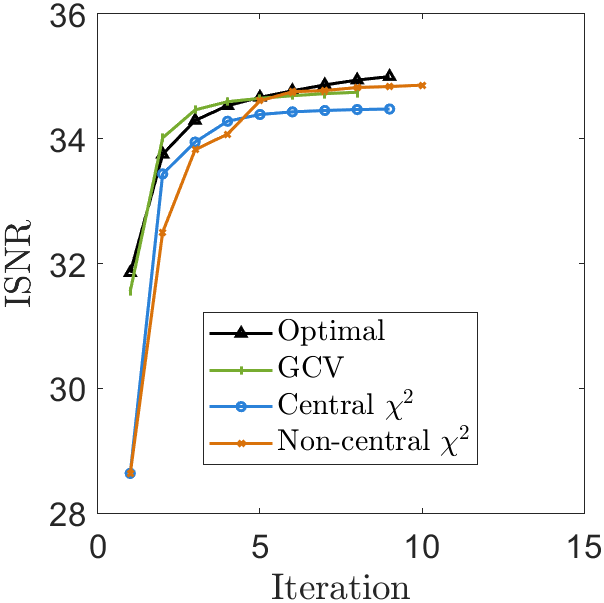}}
    \subfigure[$\text{ISNR}$: SB Wavelets \label{Ex1:SBWI}]{
\includegraphics[width=3.5cm]{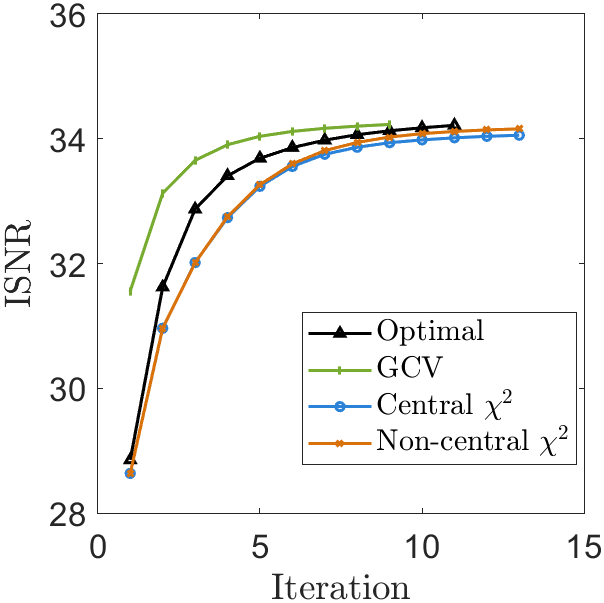}}
   \subfigure[$\text{ISNR}$: MM Framelets \label{Ex1:MMFI}]{
\includegraphics[width=3.5cm]{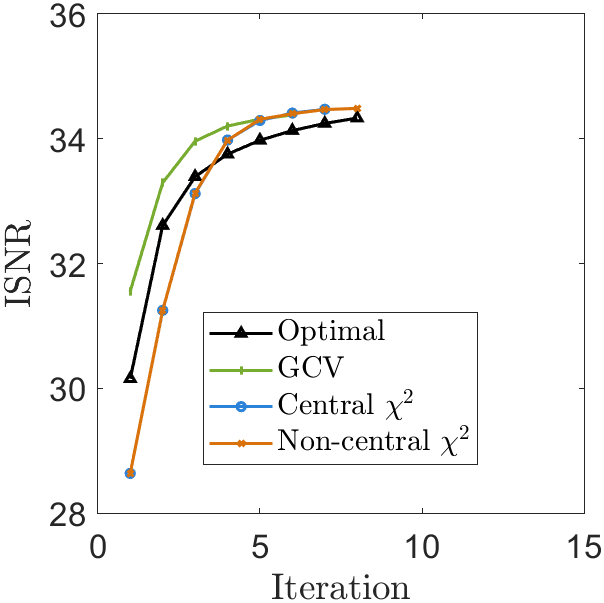}}
    \subfigure[$\text{ISNR}$: MM Wavelets\label{Ex1:MMWI}]{
\includegraphics[width=3.5cm]{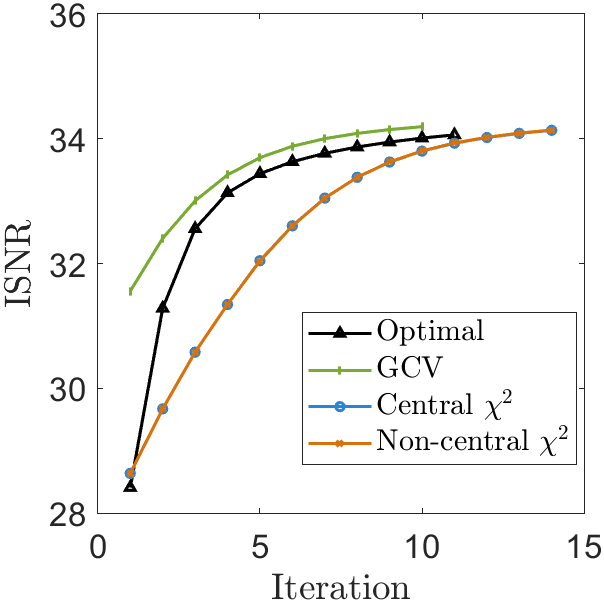}}
\caption{RE and ISNR by iteration for the methods applied to Example 1 in \cref{Ex1}. }
    \label{Ex1Conv}
\end{figure}

\begin{table}[htp]
\centering
\caption{Results for Example 1 in \cref{Ex1}.  Timing includes computing the SVD, running the iterative method to convergence, and selecting $\lambda$ at each iteration.  For the $\chi^2$ test, we report the results for the non-central $\chi^2$ test.
}
\begin{tabular}{lrrrr}
\hline
Method                  & RE   & ISNR & Iterations & Timing (s) \\ \hline
SB Framelets, Optimal  & \textbf{0.26} & \textbf{35.0} & 9         & ---       \\
SB Framelets, GCV      & \textbf{0.26} & 34.7 & 8         & 0.17       \\
SB Framelets, $\chi^2$ & \textbf{0.26} & 34.9 & 10         & 0.08       \\
SB Wavelets, Optimal   & 0.28 & 34.2 & 11         &   ---   \\
SB Wavelets, GCV       & 0.28 & 34.2 & 9          & 0.21       \\
SB Wavelets, $\chi^2$  & 0.28 & 34.2 & 13         & \textbf{0.04}       \\
MM Framelets, Optimal  & 0.28 & 34.3 & 8         & ---       \\
MM Framelets, GCV      & 0.29 & 34.4 & \textbf{6}         & 0.13       \\
MM Framelets, $\chi^2$ & 0.27 & 34.5 & 8         & 0.08       \\ 
MM Wavelets, Optimal   & 0.28 & 34.1 & 11         & ---       \\
MM Wavelets, GCV       & 0.28 & 34.2 & 10         & 0.17       \\
MM Wavelets, $\chi^2$  & 0.28 & 34.1 & 14         & 0.05       \\
\hline
\end{tabular}
\label{tab:Ex1}
\end{table}

\subsubsection{Example 2} As a larger example, we blur the image of the Hubble telescope in \cref{Ex2:x} which has size $n=512$. Here, we use zero boundary conditions, $\tilde{\sigma}_1 = 2$, $\tilde{\sigma}_2 = 8$, and $\text{band}_1 = \text{band}_2 = 50$. Gaussian noise with $\text{BSNR} = 10$ is added to the blurred image to produce $\bfb$. The PSF, $\bfb_{true}$, and $\bfb$ are shown in \cref{Ex2}. We solve this problem using SB with $\tau=0.04$ and MM with $\varepsilon = 0.03$.  The solutions are shown in \cref{Ex2R} with the RE and ISNR in \cref{Ex2Conv}.  When comparing RE and ISNR, SB and MM perform similarly for both framelets and wavelets.  We observe that framelets produce solutions with slightly smaller REs and larger ISNRs than wavelets.  With wavelets, the $\chi^2$ test does not perform as well as GCV, but with framelets, both GCV and $\chi^2$ achieve solutions near the optimal.  The results are summarized in \cref{tab:Ex2} along with timing.  From this, we can see that although the framelet methods produce better solutions, they take longer than wavelet methods to run. We also observe again that the $\chi^2$ selection method is faster than GCV, that again takes fewer iterations.

\begin{figure} 
    \centering
    \subfigure[$\bfx$\label{Ex2:x}]{
\includegraphics[width=3.5cm]{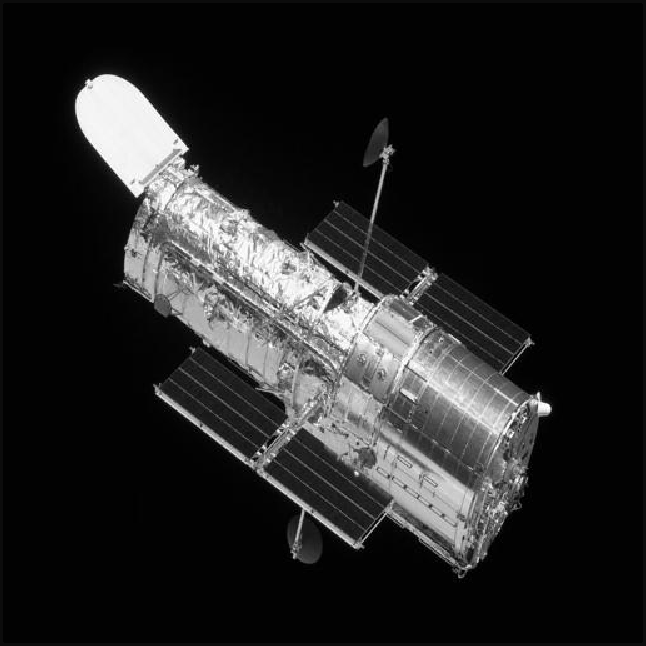}}
    \subfigure[PSF \label{Ex2:PSF}]{
\includegraphics[width=3.5cm]{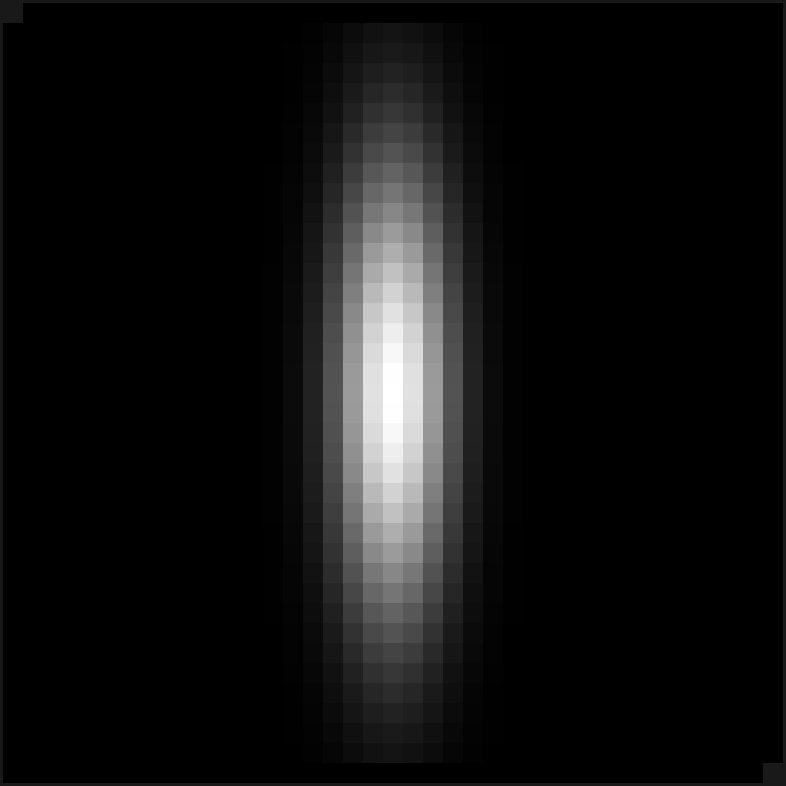}}
\subfigure[$\bfb_{true}$\label{Ex2:btrue}]{
\includegraphics[width=3.5cm]{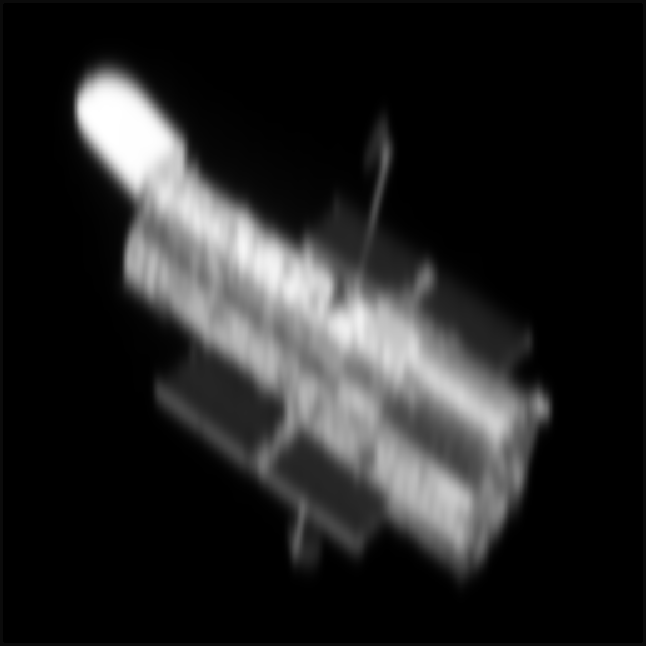}}
    \subfigure[$\bfb$\label{Ex2:b}]{
\includegraphics[width=3.5cm]{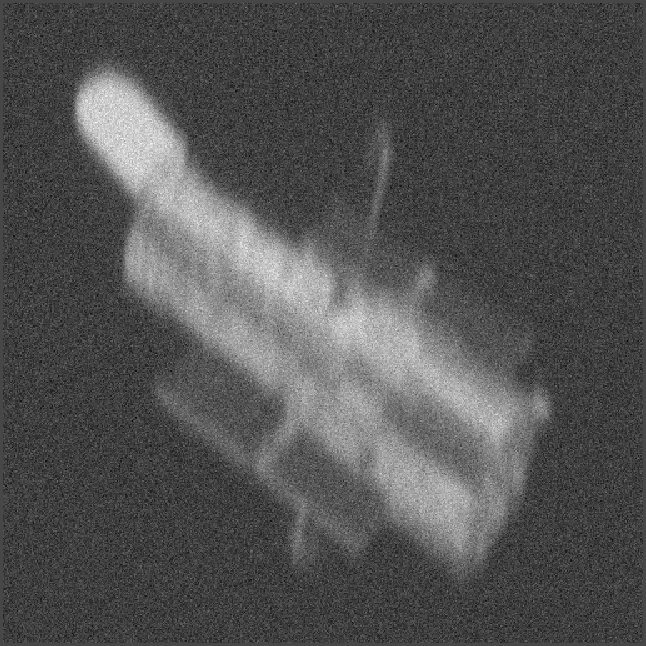}}
\caption{True image $\bfx$ ($512 \times 512$ pixels), PSF ($39 \times 39$ pixels), blurred image $\bfb_{true}$, and the blurred and noisy image $\bfb$ with $\text{BSNR}=10$ for Example 2.}
    \label{Ex2}
\end{figure}

\begin{figure}
    \centering
    \subfigure[SB Framelets, Opt\label{Ex2R:a}]{
\includegraphics[width=3.5cm]{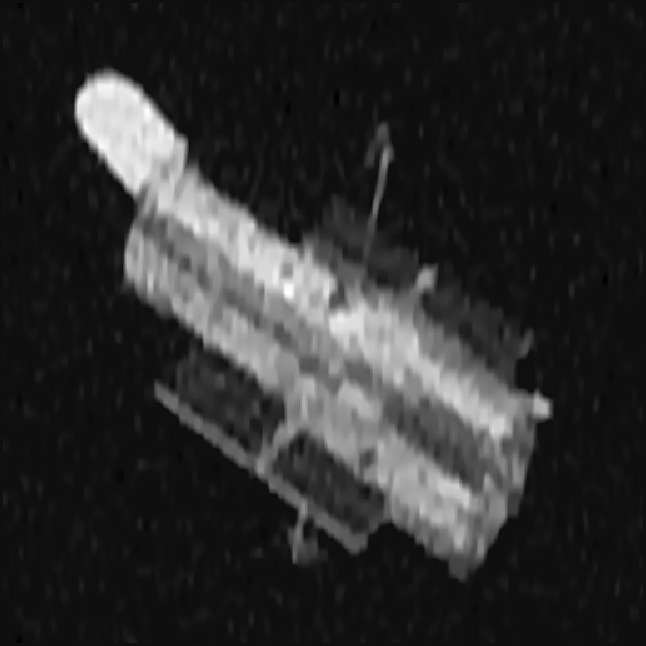}}
    \subfigure[SB Wavelets, Opt\label{Ex2R:b}]{
\includegraphics[width=3.5cm]{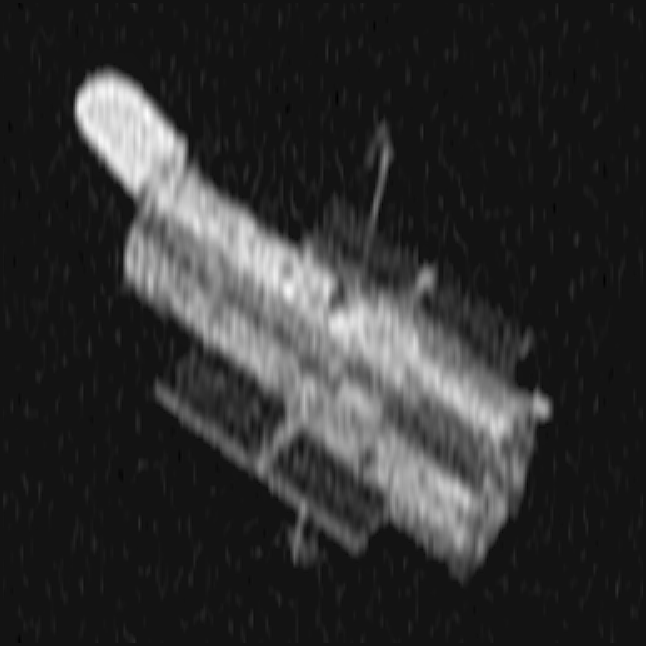}}
    \subfigure[MM Framelets, Opt\label{Ex2R:c}]{
\includegraphics[width=3.5cm]{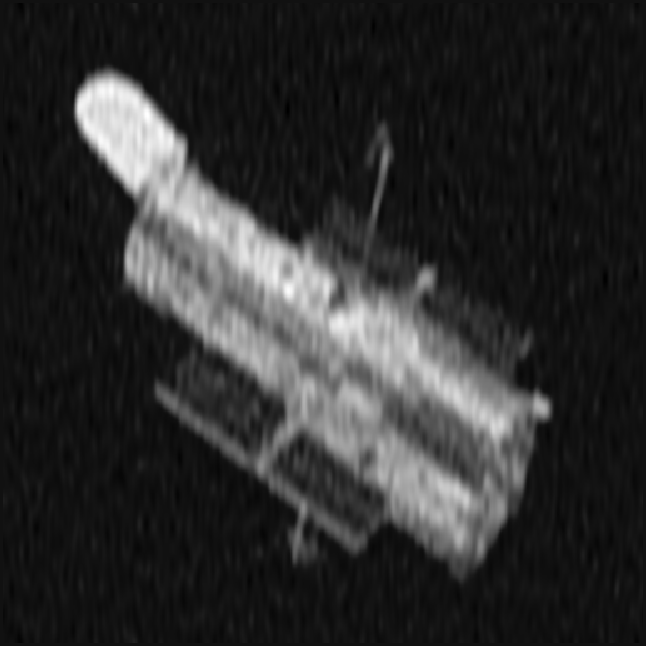}}
    \subfigure[MM Wavelets, Opt\label{Ex2R:d}]{
\includegraphics[width=3.5cm]{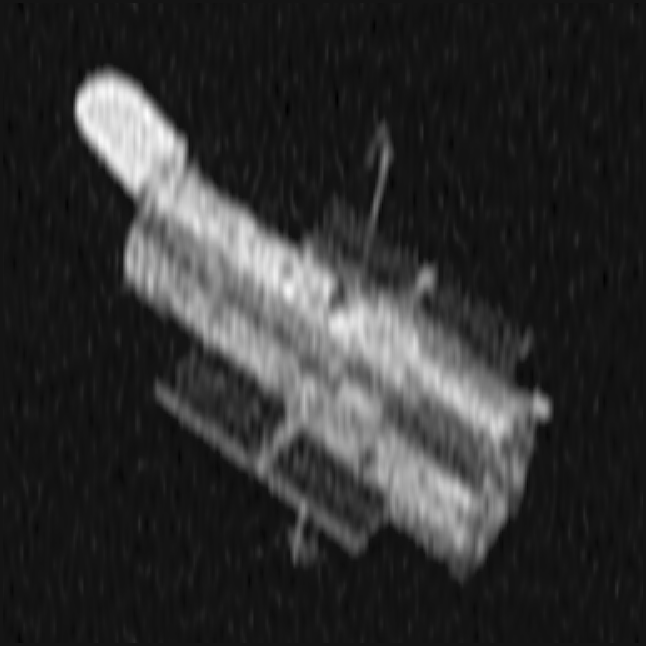}}
\subfigure[SB Framelets, GCV\label{Ex2R:e}]{
\includegraphics[width=3.5cm]{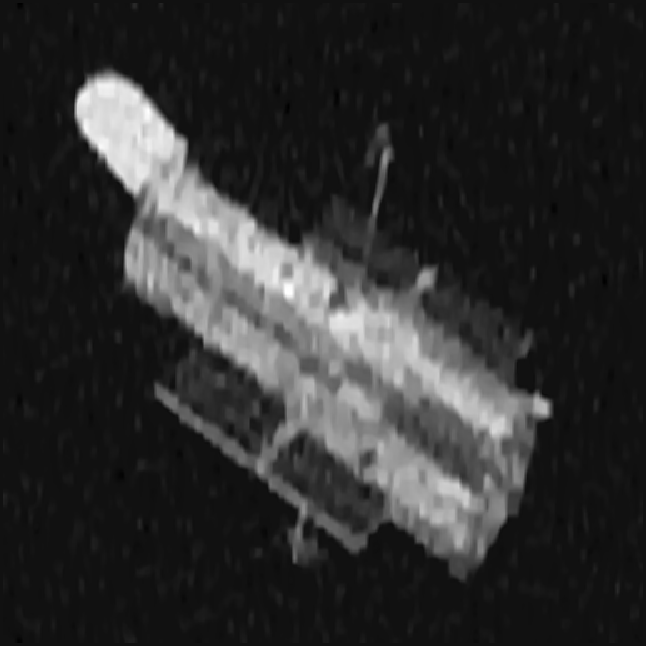}}
    \subfigure[SB Wavelets, GCV \label{Ex2R:f}]{
\includegraphics[width=3.5cm]{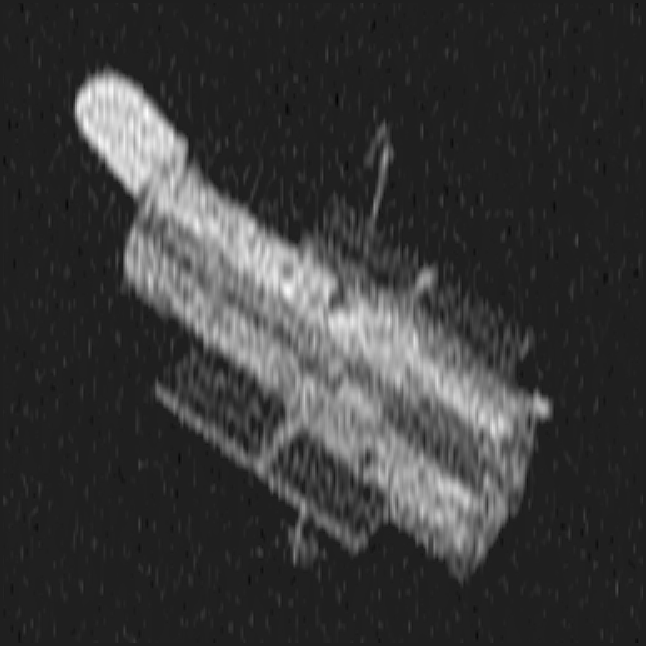}}
\subfigure[MM Framelets, GCV \label{Ex2R:g}]{
\includegraphics[width=3.5cm]{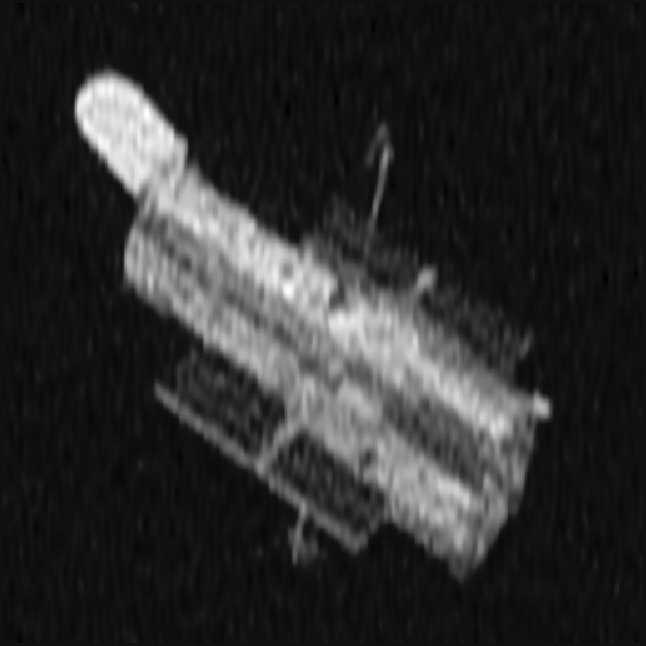}}
    \subfigure[MM Wavelets, GCV \label{Ex2R:h}]{
\includegraphics[width=3.5cm]{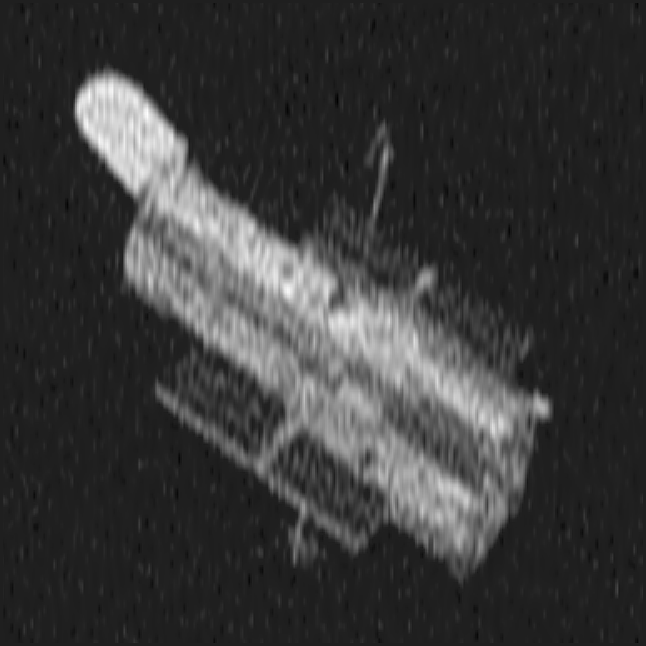}}
\subfigure[SB Framelets, $\chi^2$ \label{Ex2R:i}]{
\includegraphics[width=3.5cm]{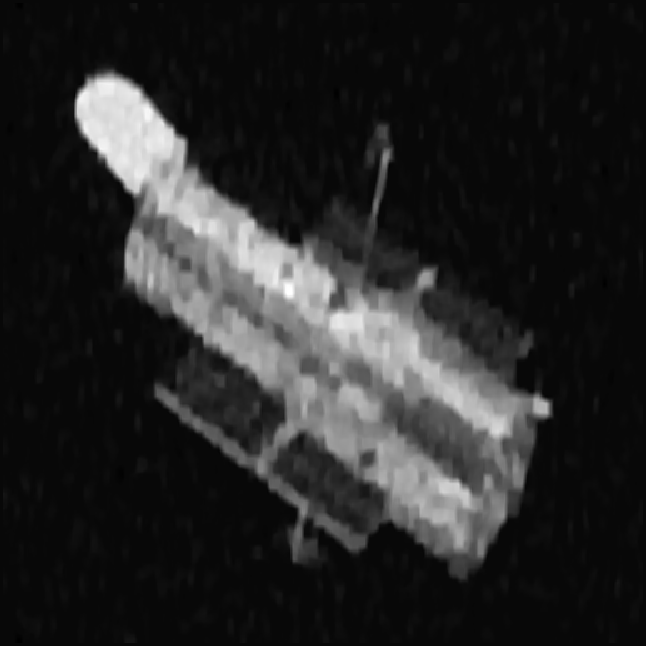}}
    \subfigure[SB Wavelets, $\chi^2$ \label{Ex2R:j}]{
\includegraphics[width=3.5cm]{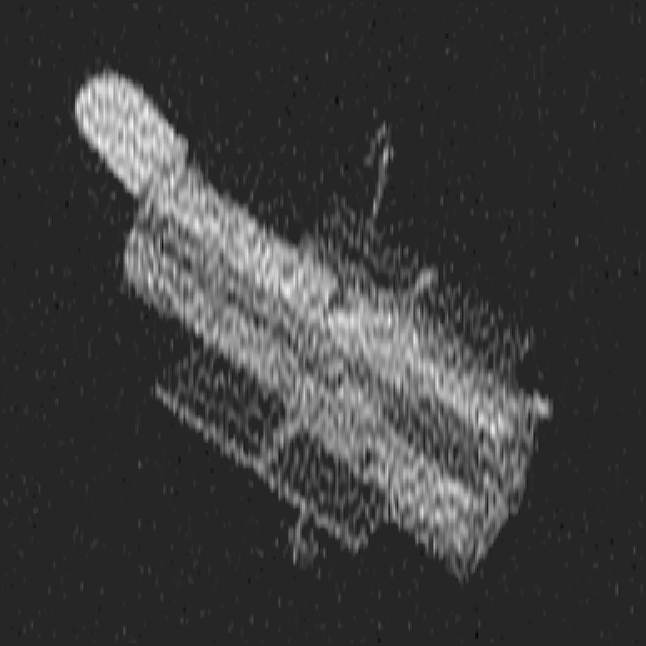}}
\subfigure[MM Framelets, $\chi^2$ \label{Ex2R:k}]{
\includegraphics[width=3.5cm]{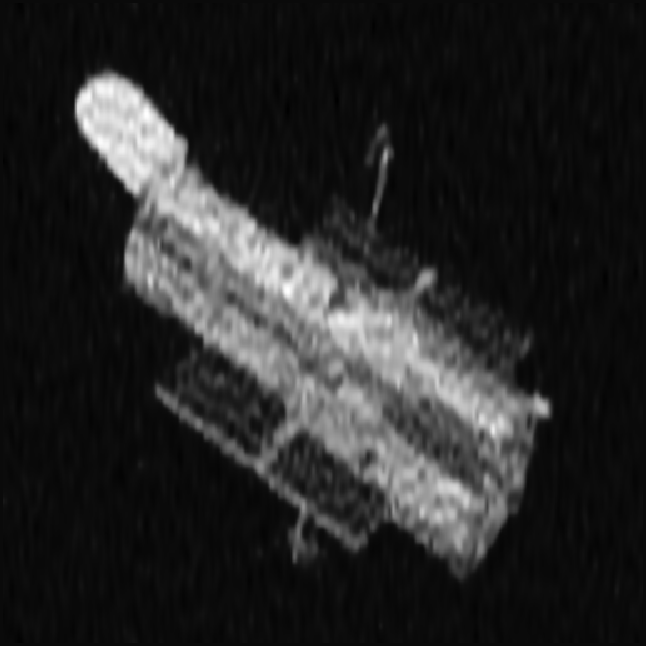}}
    \subfigure[MM Wavelets, $\chi^2$ \label{Ex2R:l}]{
\includegraphics[width=3.5cm]{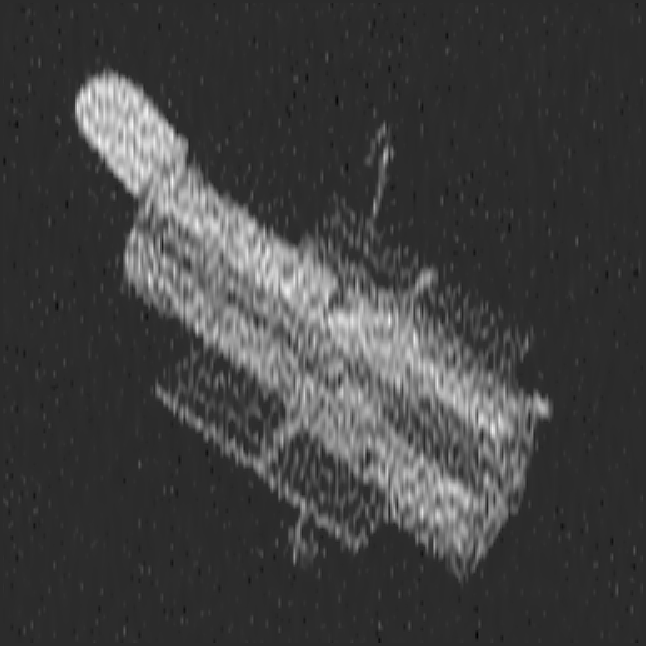}}
    \caption{Solutions to Example 2 in \cref{Ex2}. The solutions use either SB or MM to solve the problem and either framelets or wavelets for regularization.  The parameter is either fixed optimally or selected with GCV or the non-central $\chi^2$ test.}
    \label{Ex2R}
\end{figure}

\begin{figure} 
    \centering
    \subfigure[$\text{RE}$: SB Framelets \label{Ex2:SBF}]{
\includegraphics[width=3.5cm]{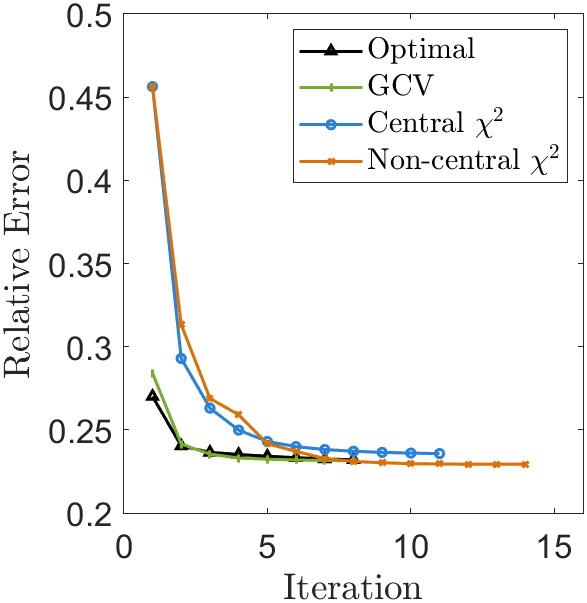}}
    \subfigure[$\text{RE}$: SB Wavelets \label{Ex2:SBW}]{
\includegraphics[width=3.5cm]{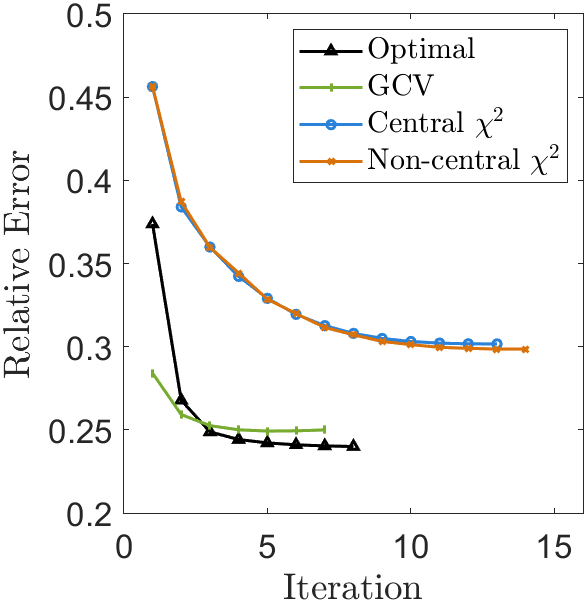}}
   \subfigure[$\text{RE}$: MM Framelets \label{Ex2:MMF}]{
\includegraphics[width=3.5cm]{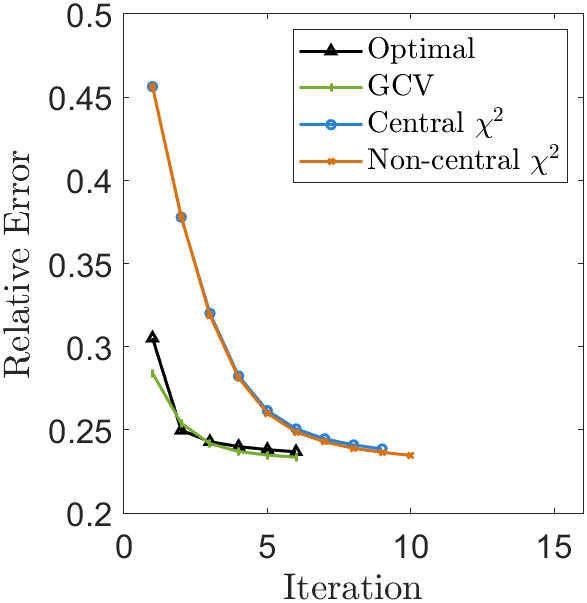}}
    \subfigure[$\text{RE}$: MM Wavelets\label{Ex2:MMW}]{
\includegraphics[width=3.5cm]{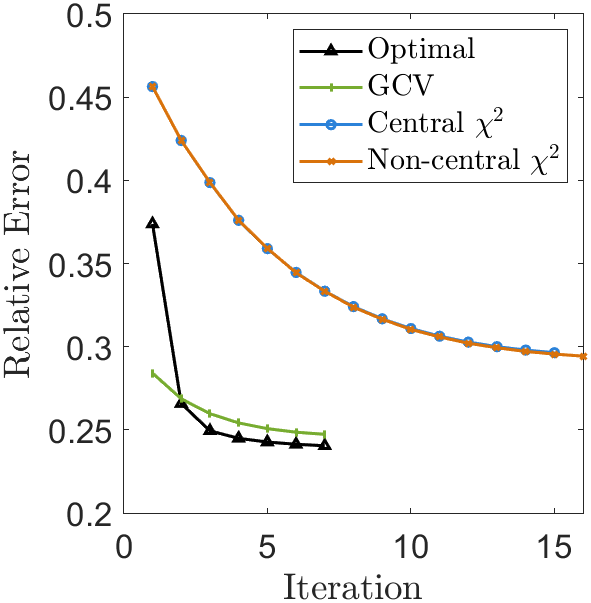}}

    \subfigure[$\text{ISNR}$: SB Framelets \label{Ex2:SBFb}]{
\includegraphics[width=3.5cm]{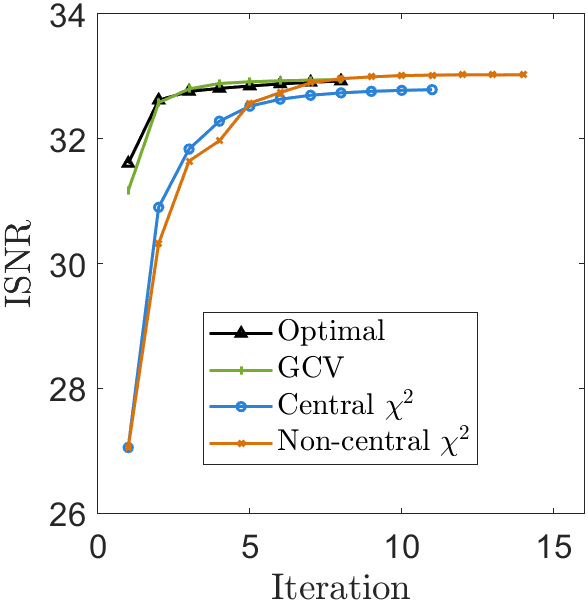}}
    \subfigure[$\text{ISNR}$: SB Wavelets \label{Ex2:SBWb}]{
\includegraphics[width=3.5cm]{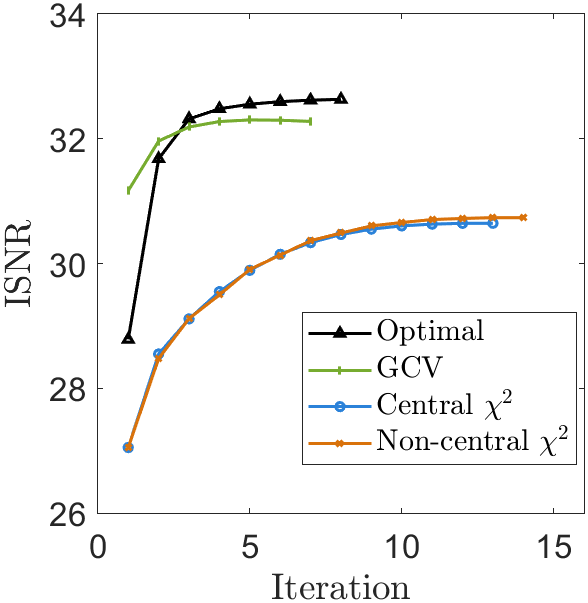}}
   \subfigure[$\text{ISNR}$: MM Framelets \label{Ex2:MMFb}]{
\includegraphics[width=3.5cm]{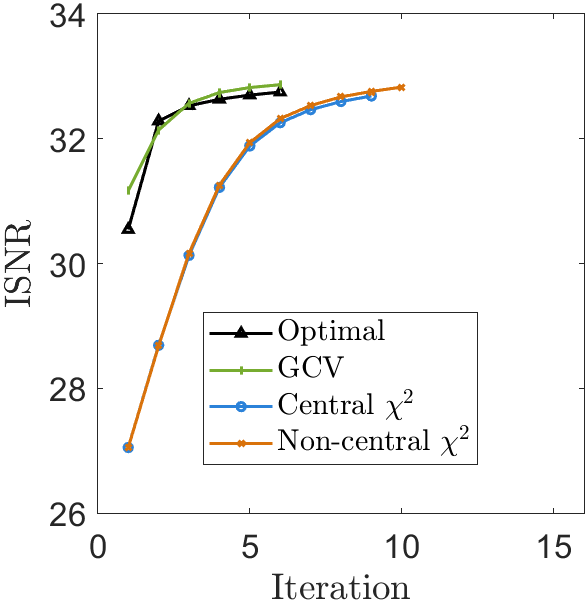}}
    \subfigure[$\text{ISNR}$: MM Wavelets\label{Ex2:MMWb}]{
\includegraphics[width=3.5cm]{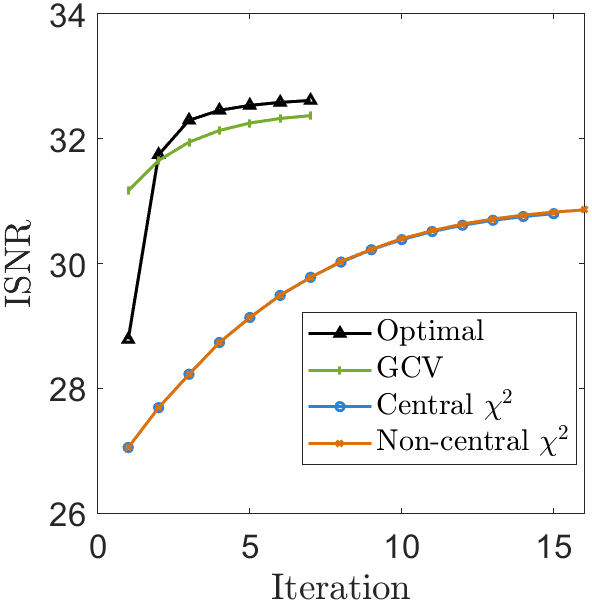}}
\caption{RE and ISNR by iteration for the methods applied to Example 2 in \cref{Ex2}.}
    \label{Ex2Conv}
\end{figure}

\begin{table}[htp]
\centering
\caption{Results for Example 2 in \cref{Ex2}. Timing includes computing the SVD, running the iterative method to convergence, and selecting $\lambda$ at each iteration. For the $\chi^2$ test, we report the results for the non-central $\chi^2$ test.}
\begin{tabular}{lrrrr}
\hline
Method                  & RE   & ISNR & Iterations & Timing (s) \\ \hline
SB Framelets, Optimal  & \textbf{0.23} & 32.9 & 8         & ---       \\
SB Framelets, GCV      & \textbf{0.23} & 32.9 & 8         & 4.76      \\
SB Framelets, $\chi^2$ & \textbf{0.23} & \textbf{33.0} & 14         & 2.75      \\
SB Wavelets, Optimal   & 0.24 & 32.6 & 8          & ---       \\
SB Wavelets, GCV       & 0.25 & 32.3 & 7          & 3.29       \\
SB Wavelets, $\chi^2$  & 0.30 & 30.7 & 14         & \textbf{1.14}       \\
MM Framelets, Optimal  & 0.24 & 32.7 & \textbf{6}         & ---       \\
MM Framelets, GCV      & \textbf{0.23} & 32.9 & \textbf{6}         & 4.22      \\
MM Framelets, $\chi^2$ & \textbf{0.23} & 32.8 & 10         & 2.81      \\ 
MM Wavelets, Optimal   & 0.24 & 32.6 & 7         & ---       \\
MM Wavelets, GCV       & 0.25 & 32.4 & 7          & 3.89       \\
MM Wavelets, $\chi^2$  & 0.29 & 30.9 & 16         & 1.65       \\
\hline
\end{tabular}
\label{tab:Ex2}
\end{table}

\subsubsection{Example 3} Now, we consider the barcode image in \cref{ExB:x}, which has size $n = 128$.  Here, we use periodic boundary conditions for the blurring matrix along with parameters $\tilde{\sigma}_1 = 1.5$, $\tilde{\sigma}_2 = 0.8$, and $\text{band}_1 = \text{band}_2 = 15$. Gaussian noise with $\text{BSNR} = 20$ is added to the blurred image to obtain $\bfb$.  The PSF, $\bfb_{true}$, and $\bfb$ are shown in \cref{ExB}. For this problem, we will compare regularization using column-orthogonal matrices with regularization using the first derivative operator in one direction. For this derivative operator, we use periodic boundary conditions, which gives a matrix of the form $\bfL_D = \bfF_{1D} \otimes \bfI_n$, where $\bfF_{1D} \in \mathbb{R}^{n \times n}$ is given by
\begin{align*}
    \bfF_{1D} = \begin{bmatrix}
        -1 & 1 & 0 &\cdots && 0\\
        0 & -1 & 1 & 0 & \cdots& 0 \\
        \vdots & & \ddots & \ddots && \vdots \\
        0 & \cdots & 0 & -1 & 1 & 0 \\
        0 & \cdots & & 0& -1 & 1 \\
        1 & 0 & \cdots & & 0 & -1
    \end{bmatrix}.
\end{align*}
This matrix only captures the derivative in the horizontal direction, which fits well with this example. We solve this problem using SB and MM where $\tau=0.02$ in SB and $\varepsilon = 0.02$ in MM.

The results for framelets and $\bfL_D$ are shown in \cref{ExBR} with the RE and ISNR plotted in \cref{ExBConv}.  Wavelet results are not shown, but the conclusions for wavelets are the same as in the previous examples. The solutions show that the derivative operator produces horizontal artifacts while the framelet solutions do not. This is reflected in the framelet methods, which have smaller REs and larger ISNRs.  For parameter selection methods, GCV and $\chi^2$ reach similar RE and ISNR for both framelets and $\bfL_D$.  With framelets, $\chi^2$ again takes more iterations but less time to converge than GCV and provides approximately comparable results in terms of RE and ISNR. With $\bfL_D$ though, MM with GCV reached the maximum number of iterations $K_{max}$. The results are summarized in \cref{tab:ExB} along with timing.  The timing shows that the $\chi^2$ selection method is faster than GCV in all cases. We also see that framelet methods take longer to run but achieve better RE and ISNR. This difference in timing comes primarily from the framelet matrix having more rows since the number of iterations for framelets and $\bfL_D$ to converge is generally comparable.
With framelets, computing the SVDs takes about 0.008 seconds compared to 0.013 seconds for the GSVD with $\bfL_D$. Within the iterative methods, however, the last $n$ columns of $\bfV_1$ and $\bfV_2$ are needed, which are of length $p=3n$ with framelets and of length $n$ with $\bfL_D$.  Similarly, $\bfh^{(k)}$ has $9n^2$ entries with framelets compared to $n^2$ with $\bfL_D$.
\begin{figure} 
    \centering
    \subfigure[$\bfx$ \label{ExB:x}]{
\includegraphics[width=3.5cm]{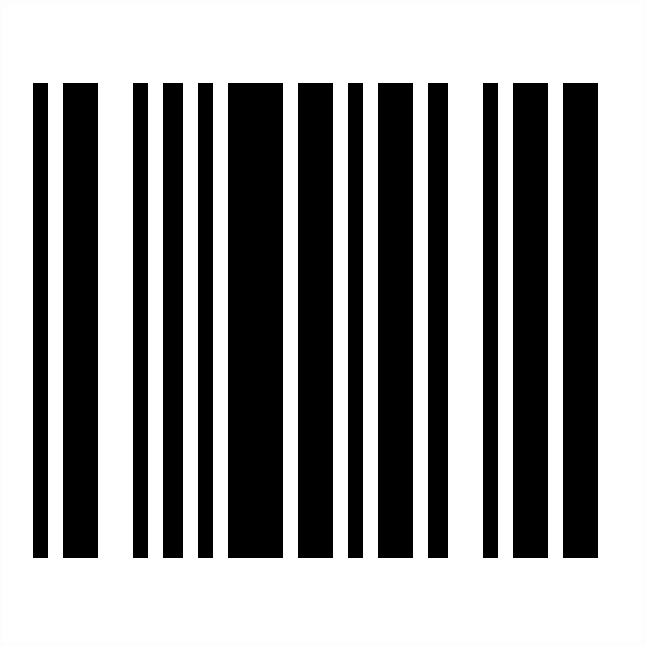}}
    \subfigure[PSF \label{ExB:PSF}]{
\includegraphics[width=3.1cm]{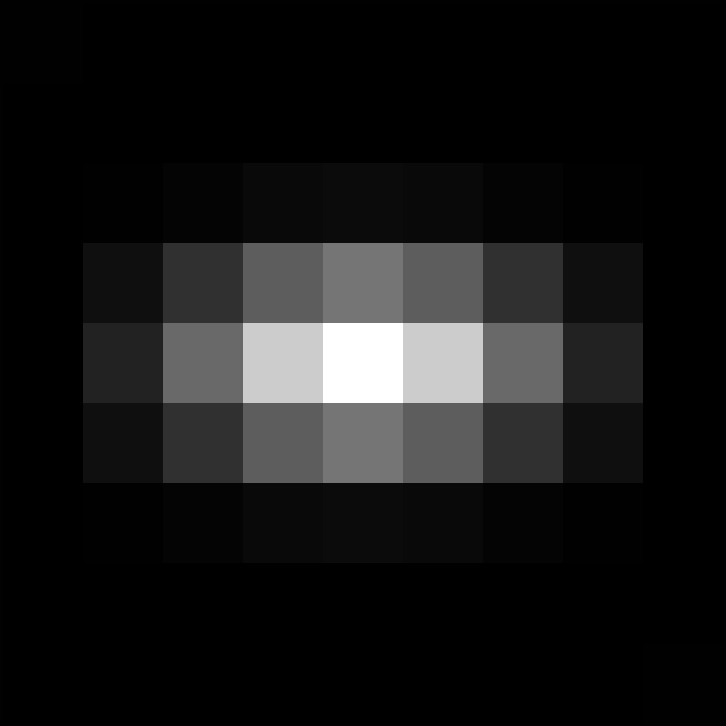}}
    \subfigure[$\bfb_{true}$ \label{ExB:btrue}]{
\includegraphics[width=3.5cm]{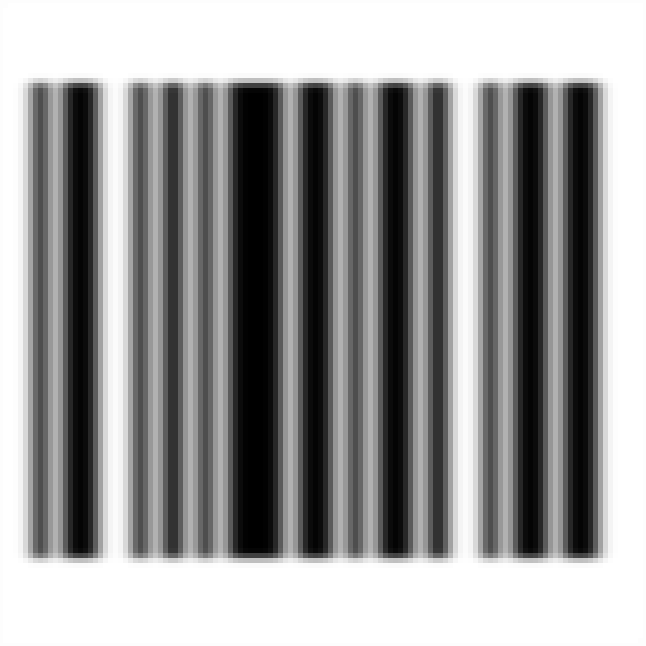}}
    \subfigure[$\bfb$ \label{ExB:b}]{
\includegraphics[width=3.5cm]{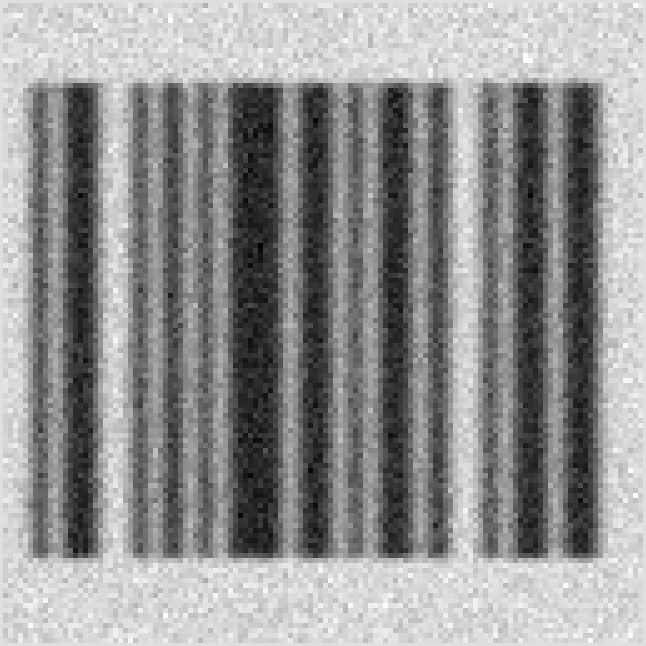}}
\caption{True image $\bfx$ ($128 \times 128$ pixels), PSF ($9 \times 9$ pixels), blurred image $\bfb_{true}$, and the blurred and noisy image $\bfb$ with $\text{BSNR}=20$ for Example 3. }
    \label{ExB}
\end{figure}

\begin{figure}
    \centering
    \subfigure[SB Framelets, Opt\label{ExBR:a}]{
\includegraphics[width=3.5cm]{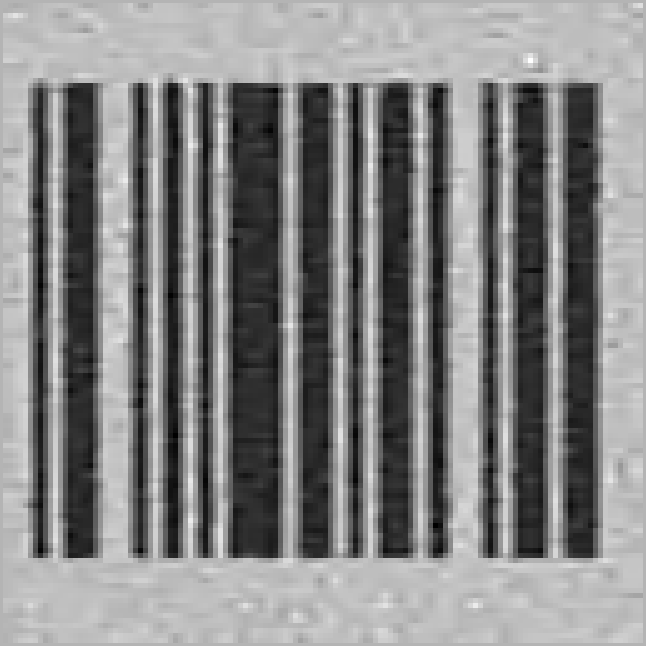}}
\subfigure[SB $\bfL_D$, Opt\label{ExBR:b}]{
\includegraphics[width=3.5cm]{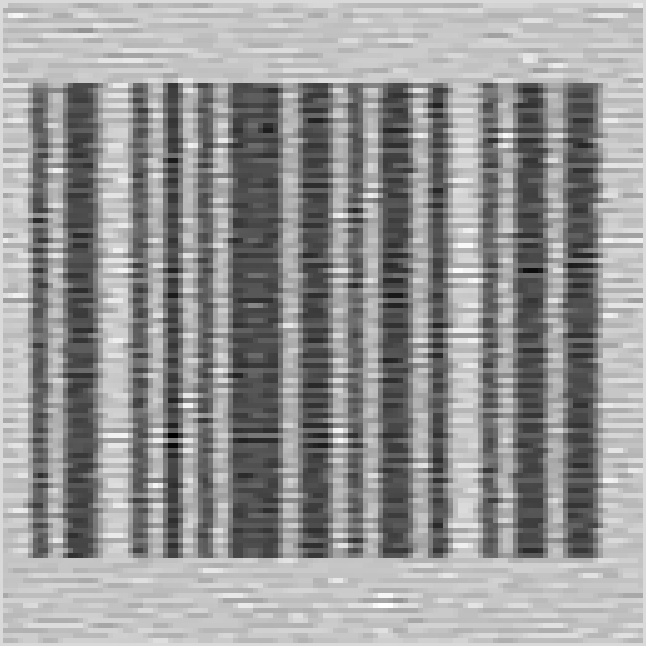}}
    \subfigure[MM Framelets, Opt\label{ExBR:c}]{
\includegraphics[width=3.5cm]{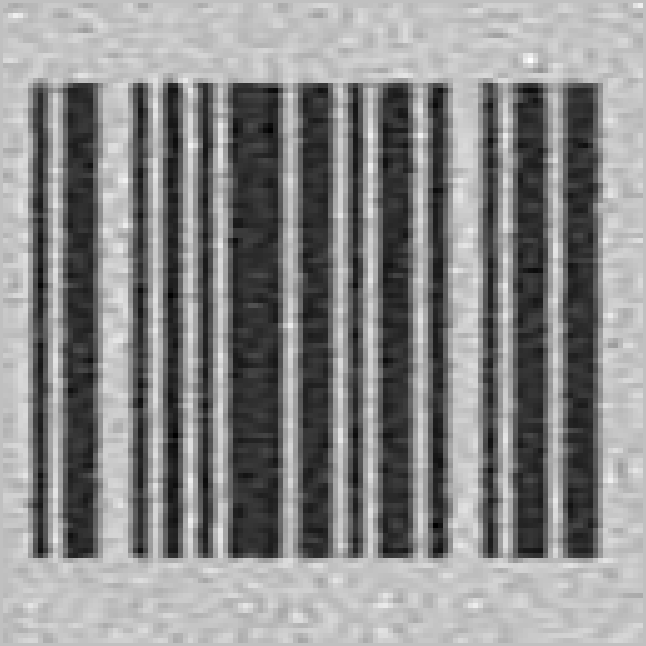}}    
\subfigure[MM $\bfL_D$, Opt\label{ExBR:d}]{
\includegraphics[width=3.5cm]{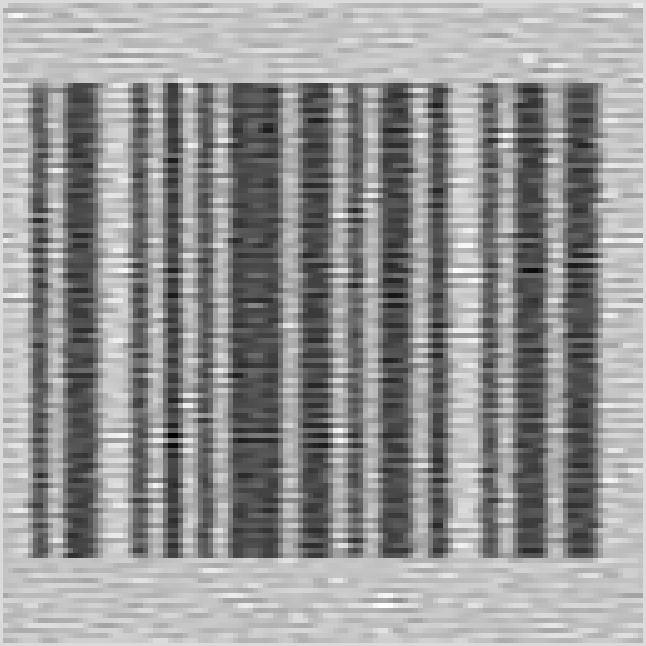}}
\subfigure[SB Framelets, GCV\label{ExBR:e}]{
\includegraphics[width=3.5cm]{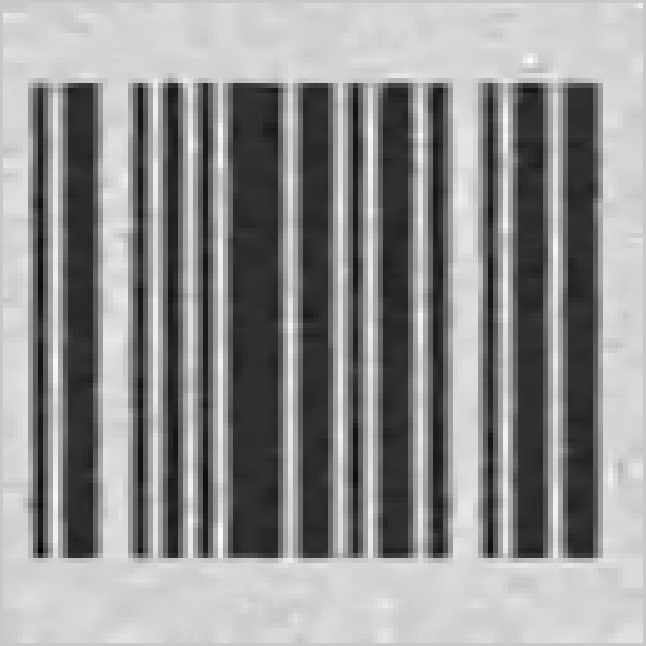}}
    \subfigure[SB $\bfL_D$, GCV \label{ExBR:f}]{
\includegraphics[width=3.5cm]{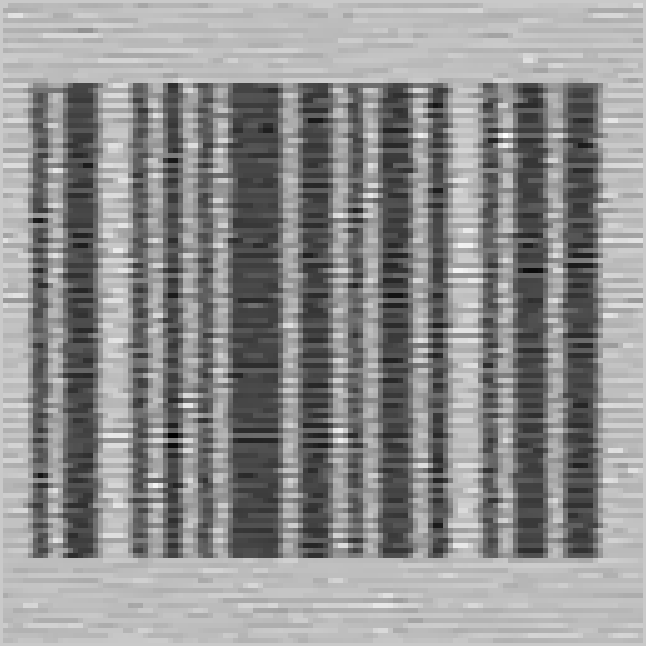}}
\subfigure[MM Framelets, GCV\label{ExBR:g}]{
\includegraphics[width=3.5cm]{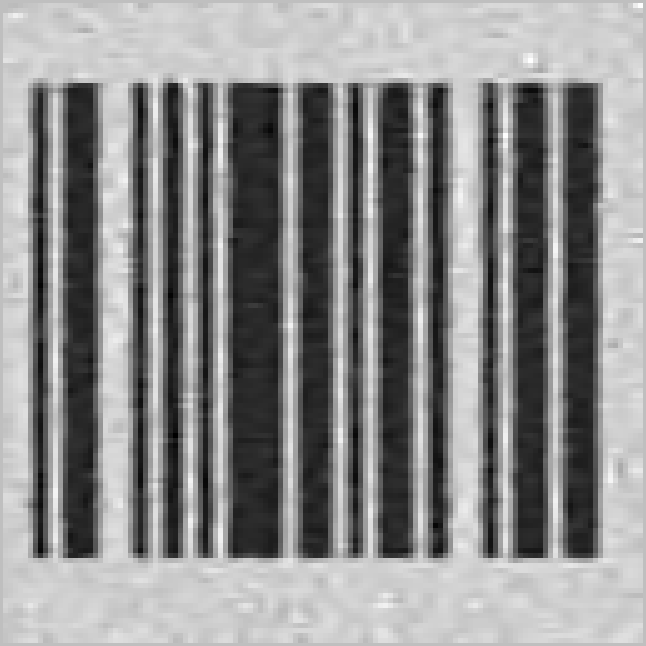}}
    \subfigure[MM $\bfL_D$, GCV \label{ExBR:h}]{
\includegraphics[width=3.5cm]{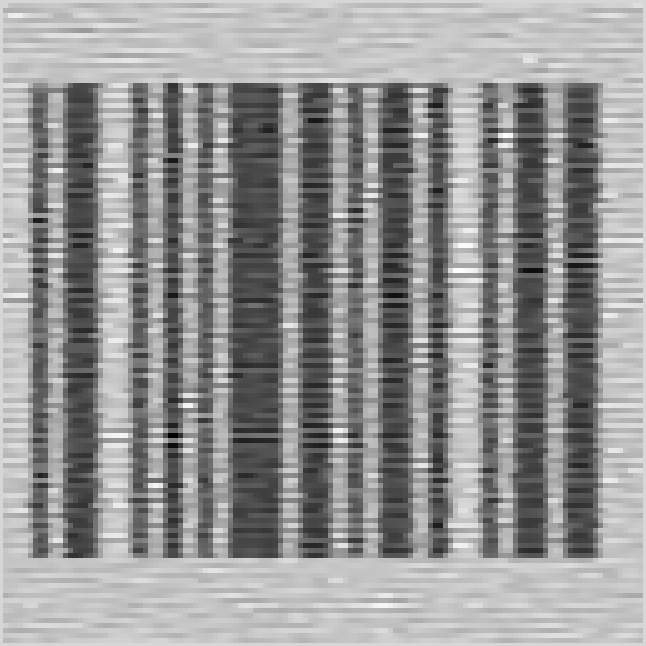}}
   \subfigure[SB Framelets, $\chi^2$ \label{ExBR:i}]{
\includegraphics[width=3.5cm]{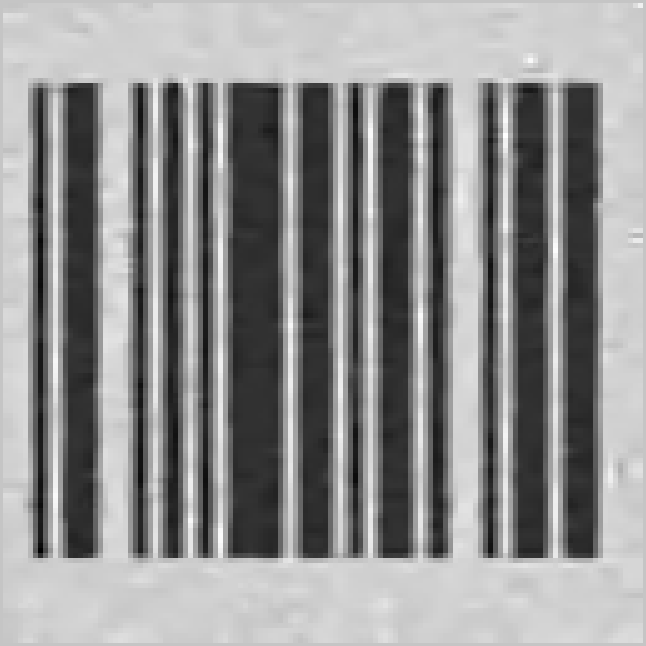}}
    \subfigure[SB $\bfL_D$, $\chi^2$ \label{ExBR:j}]{
\includegraphics[width=3.5cm]{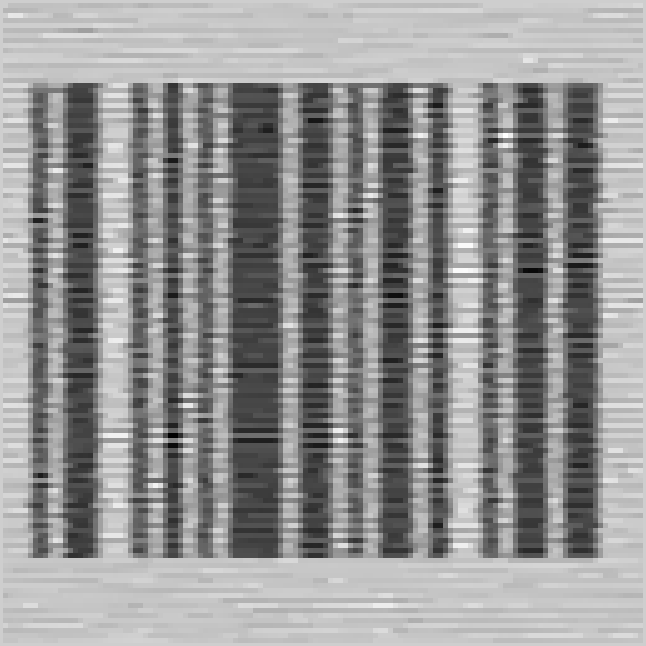}}
\subfigure[MM Framelets, $\chi^2$\label{ExBR:k}]{
\includegraphics[width=3.5cm]{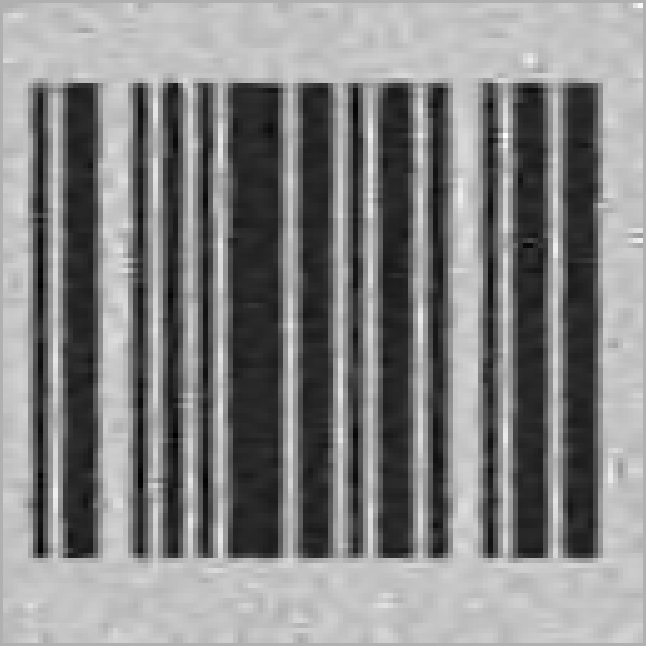}}
    \subfigure[MM $\bfL_D$, $\chi^2$ \label{ExBR:l}]{
\includegraphics[width=3.5cm]{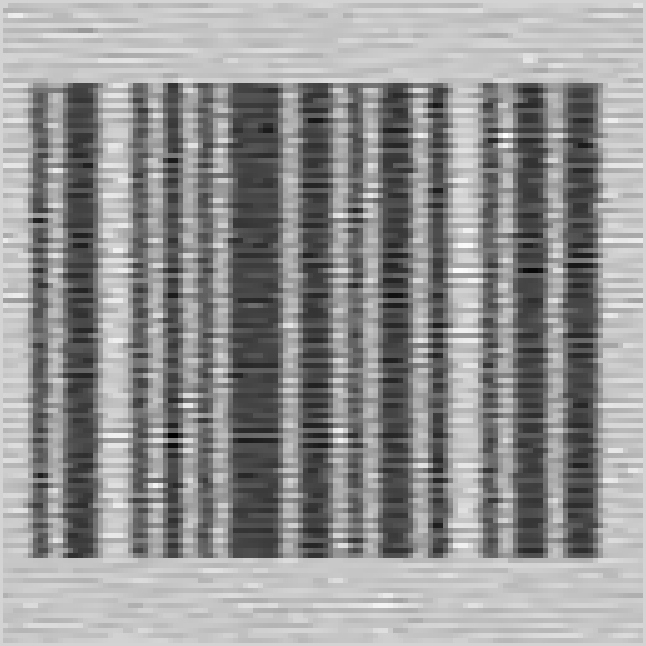}}
    \caption{Solutions to Example 3 in \cref{ExB}.  The solutions use either SB or MM to solve the problem and either framelets or $\bfL_D$ for regularization.  The parameter is either fixed optimally or selected with GCV or the non-central $\chi^2$ test.}
    \label{ExBR}
\end{figure}

\begin{figure} 
    \centering
    \subfigure[$\text{RE}$: SB Framelets \label{ExB:SBF}]{
\includegraphics[width=3.5cm]{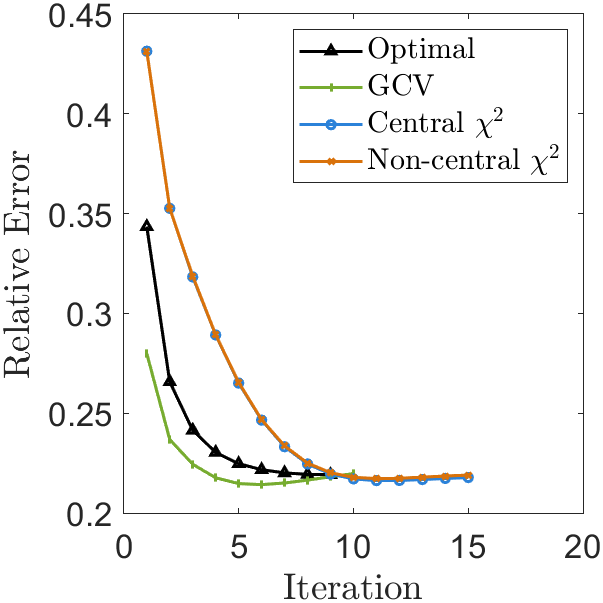}}
    \subfigure[$\text{RE}$: SB $\bfL_D$ \label{ExB:SBW}]{
\includegraphics[width=3.5cm]{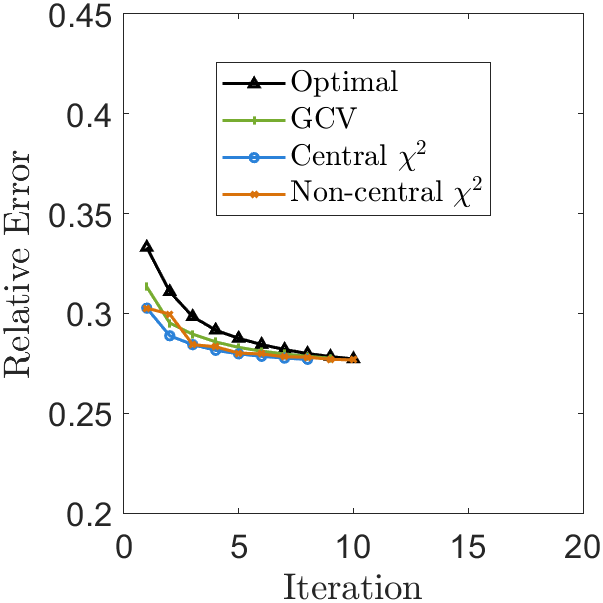}}
   \subfigure[$\text{RE}$: MM Framelets \label{ExB:MMF}]{
\includegraphics[width=3.5cm]{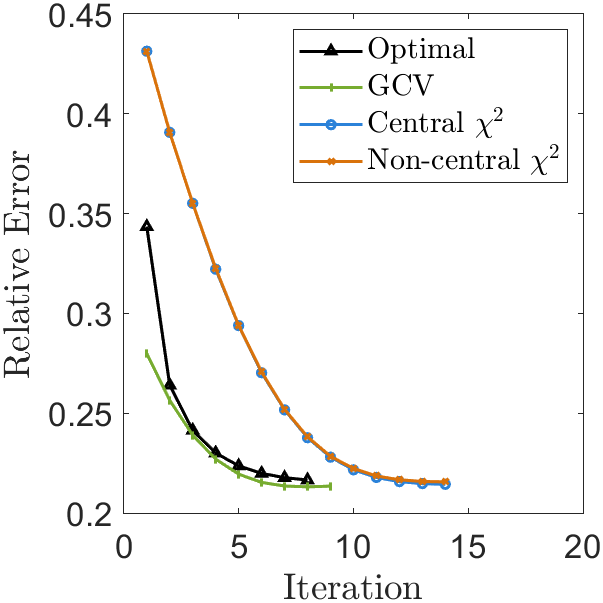}}
    \subfigure[$\text{RE}$: MM $\bfL_D$\label{ExB:MMW}]{
\includegraphics[width=3.5cm]{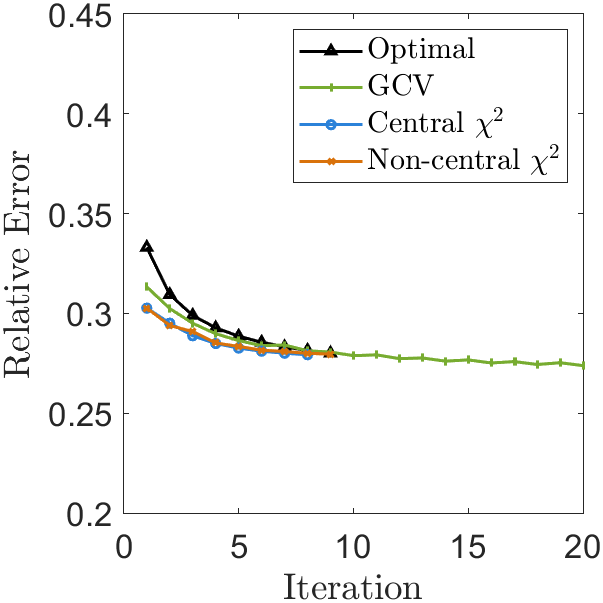}}

    \subfigure[$\text{ISNR}$: SB Framelets \label{ExB:SBFI}]{
\includegraphics[width=3.5cm]{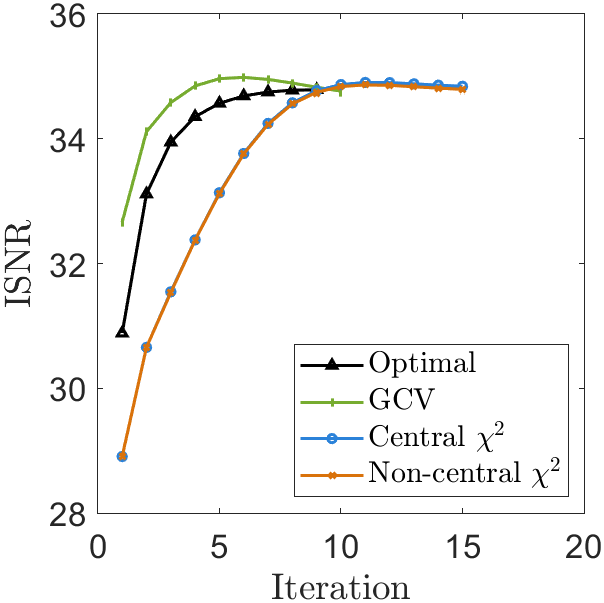}}
    \subfigure[$\text{ISNR}$: SB $\bfL_D$ \label{ExB:SBWI}]{
\includegraphics[width=3.5cm]{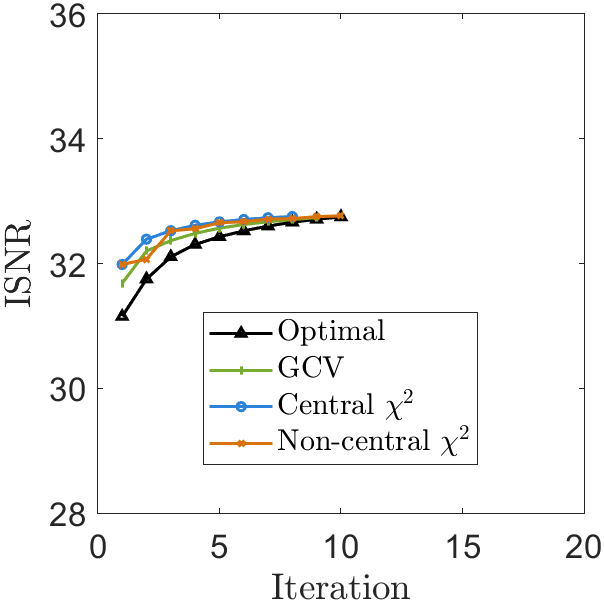}}
   \subfigure[$\text{ISNR}$: MM Framelets \label{ExB:MMFI}]{
\includegraphics[width=3.5cm]{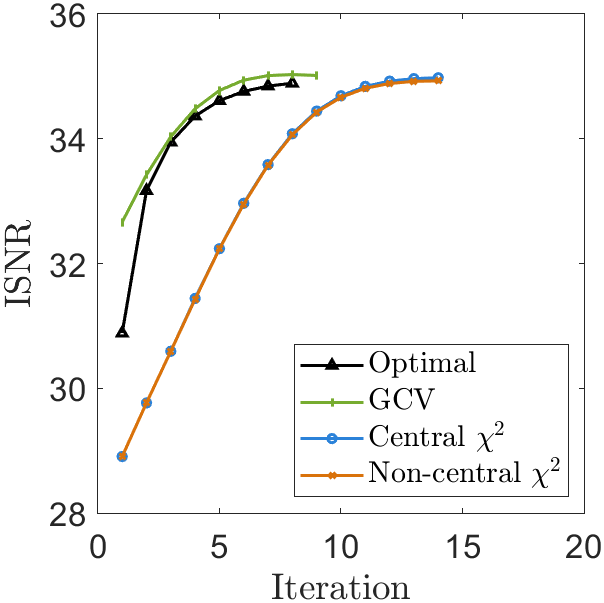}}
    \subfigure[$\text{ISNR}$: MM $\bfL_D$ \label{ExB:MMWI}]{
\includegraphics[width=3.5cm]{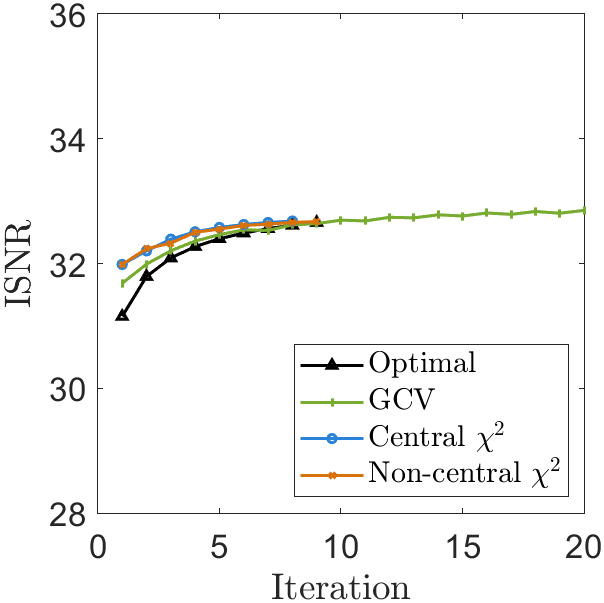}}
\caption{RE and ISNR by iteration for the methods applied to Example 3 in \cref{ExB}.}
    \label{ExBConv}
\end{figure}

\begin{table}[htp]
\centering
\caption{Results for Example 3 in \cref{ExB}.  Timing includes computing the decomposition (SVD for  framelets, GSVD for $\bfL_D$), running the iterative method to convergence, and selecting $\lambda$ at each iteration. For the $\chi^2$ test, we report the results for the non-central $\chi^2$ test.}
\begin{tabular}{lrrrr}
\hline
Method                  & \multicolumn{1}{c}{RE}   & ISNR & Iterations & Timing (s) \\ \hline
SB Framelets, Optimal  & 0.22 & 34.8 & 9         & ---       \\
SB Framelets, GCV      & 0.22 & 34.8 & 10         & 0.22      \\
SB Framelets, $\chi^2$ & 0.22 & 34.8 & 15         &  0.08     \\
SB $\bfL_D$, Optimal  & 0.28 & 32.8 & 10         & ---       \\
SB $\bfL_D$, GCV      & 0.28 & 32.7 & 9         & 0.17      \\
SB $\bfL_D$, $\chi^2$ & 0.28 & 32.8 & 10         & 0.04      \\
MM Framelets, Optimal  & 0.22 & 34.9 & \textbf{8}         & ---       \\
MM Framelets, GCV      & \textbf{0.21} & \textbf{35.0} & 9         &  0.23     \\
MM Framelets, $\chi^2$ & 0.22 & 34.9 & 14         &  0.14     \\ 
MM $\bfL_D$, Optimal  & 0.28 & 32.7 & 9         &  ---      \\
MM $\bfL_D$, GCV      & 0.27 & 32.9 & 20         &  0.40     \\
MM $\bfL_D$, $\chi^2$ & 0.28 & 32.7 & 9         &\textbf{0.03}      \\ \hline
\end{tabular}
\label{tab:ExB}
\end{table}

\section{Conclusions} \label{sec:conc}

In this paper, we presented decomposition-based algorithms for solving an $\ell_1$-regularized problem with KP-structured matrices.  Within these algorithms, we utilized the column orthogonality of a family of regularization operators, which include framelets and wavelets.  These operators allow us to obtain a joint decomposition using SVDs rather than GSVDs, making them more efficient than matrices of comparable size without this property, such as finite-difference operators.  When tested on numerical examples, framelet-based regularization methods are slower than wavelet-based ones but generally produce better solutions in terms of RE and ISNR.  For the parameter selection methods, GCV and $\chi^2$ both perform well with framelets, with $\chi^2$ being faster than GCV. Overall, the results support choosing the SB algorithm with  framelets for regularization and the $\chi^2$ parameter selection method, as  optimal in the context of the solution of \cref{eq:Lxh} and in relation to the other algorithm choices presented here.  We also make the interesting observation that the solution of the Generalized Tikhonov problem with a column-orthogonal regularization matrix is just the solution of the standard Tikhonov problem, independent of the choice of this column-orthogonal matrix. In the future, these decomposition-based algorithms could also be extended for $\ell_p$ regularization.

\section*{Acknowledgments}
Funding: This work was partially supported by the National Science Foundation (NSF) under grant  DMS-2152704 for  Renaut.  Any opinions, findings, conclusions, or recommendations expressed in this material are those of the authors and do not necessarily reflect the views of the National Science Foundation. M.I. Espa\~nol was supported through a Karen Uhlenbeck EDGE Fellowship.

\bibliographystyle{siamplain}

\bibliography{references}

\end{document}